\documentclass{amsart}
\usepackage{amsmath}
\usepackage{amsthm}
\usepackage{amssymb}
\usepackage{tikz}
\usetikzlibrary{automata,positioning}
\usepackage{enumitem}
\usepackage{mathrsfs}

\newtheorem{theorem}{Theorem}[section]

\newtheorem{alphatheorem}{Theorem}

\newtheorem{alphaconjecture}{Conjecture}

\theoremstyle{definition}

\newtheorem{definition}[theorem]{Definition}
\newtheorem{proposition}[theorem]{Proposition}
\newtheorem{step}{Step}

\newtheorem*{proposition*}{Proposition}
\newtheorem*{observation*}{Observation}
\newtheorem*{claim*}{Claim}

\newtheorem*{lemma*}{Lemma}
\newtheorem{example}[theorem]{Example}
\newtheorem{corollary}[theorem]{Corollary}
\newtheorem{remark}[theorem]{Remark}
\newtheorem{conjecture}[theorem]{Conjecture}
\newtheorem*{conjecture*}{Conjecture}

\newtheorem*{convention*}{Convention}

\newtheorem{question}{Question}

\theoremstyle{plain}
\newtheorem{lemma}[theorem]{Lemma}

\newcommand{\bra}[1]{ \left( #1 \right) }

\renewcommand{\tilde}{\widetilde}

\newcommand{\abs}[1]{\left|#1\right|}
\newcommand{\fp}[1]{\left\{ #1 \right\}}

\newcommand{\ip}[1]{\left[ #1 \right]}

\newcommand{\norm}[1]{\left\lVert #1 \right\rVert}

\renewcommand{\b}{\beta}

\newcommand{\fpa}[1]{\left\lVert #1 \right\rVert_{\mathbb{R}/\mathbb{Z}}}
\newcommand{\e}{\varepsilon}
\renewcommand{\a}{\alpha}
\renewcommand{\b}{\beta}

\usepackage{mathrsfs}
\usepackage[mathscr]{euscript}
	
\renewcommand{\b}{\beta}

\newcommand{\NN}{\mathbb{N}}
\newcommand{\QQ}{\mathbb{Q}}
\newcommand{\Q}{\QQ}

\newcommand{\ZZ}{\mathbb{Z}}
\newcommand{\RR}{\mathbb{R}}
\newcommand{\R}{\RR}
\newcommand{\N}{\NN}
\newcommand{\Z}{\ZZ}
\newcommand{\CC}{\mathbb{C}}

\newcommand{\TT}{\mathbb{T}}

\newcommand{\cB}{\mathcal{B}}
\newcommand{\cF}{\mathcal{F}}
\newcommand{\cX}{\mathscr{X}}
\newcommand{\cG}{\mathcal{G}}

\newcommand{\floor}[1]{\left\lfloor #1 \right\rfloor}

\newcommand{\set}[2]{\left\{ #1 \ \middle| \ #2 \right\} }

\newcommand{\parbreak}[1]{
\begin{center}
***
\end{center}
}

\newcommand{\FS}{\operatorname{FS}}

\newcommand{\p}{p}

\newcommand{\cFe}{\cF_\emptyset}

\newcommand{\IP}{$\operatorname{IP}$}
\newcommand{\IPS}{$\mathrm{IPS}$}
\newcommand{\IPP}{$\mathrm{IP}_+$}

\newcommand{\IPrich}{\IP-rich}
\newcommand{\IPPrich}{\IPP-rich}
\newcommand{\IPSrich}{\IPS-rich}

\newcommand{\IPd}{$\operatorname{IP}^*$}

\newcommand{\comment}[1]{%
}

\newcommand{\inter}[1]{\operatorname{int}#1}
\newcommand{\cl}[1]{\operatorname{cl}#1}
\newcommand{\GP}{\mathrm{GP}}

\usepackage{soul}

\usepackage{stmaryrd}
\newcommand{\ifbra}[1]{\left\llbracket #1 \right\rrbracket}

\newcommand{\ceil}[1]{\left\lceil #1 \right\rceil}

\newcommand{\lbr}{\langle\!\langle}

\newcommand{\rbr}{\rangle\!\rangle}

\newcommand{\Haland}{H\r{a}land}
\newcommand{\Malcev}{Mal'cev}

\begin{document}

\address{Department of Mathematics and Computer Science\\Institute of Mathematics\\
Jagiellonian University\\
ul. prof. Stanis\l{}awa \L{}ojasiewicza 6\\
30-348 Krak\'{o}w}
\email{jakub.byszewski@gmail.com}

\author[J. Byszewski \and J. Konieczny]{Jakub Byszewski \and Jakub Konieczny
}
\address{Mathematical Institute \\ 
University of Oxford\\
Andrew Wiles Building \\
Radcliffe Observatory Quarter\\
Woodstock Road\\
Oxford\\
OX2 6GG}
\email{jakub.konieczny@gmail.com}


\title{Automatic sequences, generalised polynomials, and nilmanifolds}


\maketitle 

\begin{abstract}
	We conjecture that bounded generalised polynomial functions cannot be generated by finite automata, except for the trivial case when they are periodic away from a finite set.
	
	 Using methods from ergodic theory, we are able to partially resolve this conjecture, proving that any hypothetical counterexample is periodic away from a very sparse and structured set. 
	 In particular, we show that for a polynomial $p(n)$ with at least one irrational coefficient (except for the constant one) and integer $m$, the sequence $\floor{p(n)} \bmod{m}$ is never automatic.  
	 
	 	 We also obtain a conditional result, where we prove the conjecture under the assumption that the characteristic sequence of the set of powers of an integer $k\geq 2$ is not given by a generalised polynomial.
	 
\end{abstract}

\section*{Introduction}

Automatic sequences are sequences whose $n$-th term is produced by a finite state machine from base $k$ digits of $n$. (A precise definition is given below.) By definition, automatic sequences can take only finitely many values. Allouche and Shallit \cite{Allouche-Shallit-1992, AlloucheShallit-2003} have generalized the notion of automatic sequences to a wider class of regular sequences and demonstrated their ubiquity and links with multiple branches of mathematics and computer science. The problem of demonstrating that a certain sequence is or is not automatic or regular has been widely studied, particularly for sequences of arithmetic origin. The aim of this article is to continue this study for sequences that arise from generalized polynomials, i.e. expressions involving algebraic operations and the floor function via dynamical and ergodic methods. This is possible because by the work of Bergelson and Leibman generalized polynomials are strongly related to dynamics on nilmanifolds. The results obtained lead to a number of interesting questions concerning zero sets of generalized polynomials that we hope will be of independent interest.

In \cite[Theorem 6.2]{AlloucheShallit-2003} it is proved that the sequence $(f(n))_{n \geq 0}$ given by $f(n)= \lfloor \alpha n + \beta\rfloor$ for real numbers $\alpha, \beta$ is regular if and only if $\alpha$ is rational. The method used there does not immediately generalise to higher degree polynomials in $n$, but the proof implicitly uses rotation on a circle by an angle of $2\pi \alpha$. Replacing the rotation on a circle by a skew product transformation on a torus (as in Furstenberg's proof of Weyl's equidistribution theorem), we easily obtain the following result. (For more on regular sequences, see Section \ref{sec:DEF}).

\begin{alphatheorem}\label{thm:main-sortof} Let $p\in \R[x]$ be a polynomial. Then the sequence $f(n)=\lfloor p(n) \rfloor, n\geq 0$ is regular if and only if all the coefficients of $p$ except possibly for the constant term are rational.
\end{alphatheorem}

In fact, we show the stronger property that for an integer $m\geq 2$ the sequence $f(n) \bmod m$ is not automatic unless all the coefficients of $p$ except for the constant term are rational, in which case it is periodic. It is natural to inquire whether a similar result can be proven for more complicated expressions involving the floor function such as e.g. $f(n)= \lfloor \alpha \lfloor \beta n^2 + \gamma\rfloor^2 + \delta n +\varepsilon  \rfloor$. Such sequences are called generalized polynomial  and have been intensely studied. The main motivation for this project is the following conjecture.

\begin{alphaconjecture}\label{conjecture:main}
	Suppose that a sequence $f$ is simultaneously automatic and generalised polynomial. Then $f$ is ultimately periodic. 
\end{alphaconjecture}

(We say that a sequence $f$ is \emph{ultimately periodic} if it coincides with a periodic sequence except at a finite set.)


We are able to partially resolve this conjecture. First of all, we prove that the conjecture holds except on a set of density zero. In fact, in order to obtain such a result, we only need a specific property of automatic sequences. For the purpose of stating the next theorem, let us say that a sequence $f \colon \NN \to X$ is \emph{weakly periodic} if for any restriction of $f$ to an arithmetic sequence, $f'(n) = f(a n + b)$, $a \in \NN,\ b \in \NN_0$, there exist $q \in \NN$, $r, s \in \NN_0$ with $r \neq s$, such that $f'(q n + r) = f'(q n + s)$. Of course, any periodic sequence is weakly periodic, but not conversely. All automatic sequences are weakly periodic, which follows from the finiteness of kernels (see Lemma \ref{lem:auto=>weak-per}). Another non-trivial example is the characteristic function of the square-free numbers.

\begin{alphatheorem}\label{thm:main-weakly-periodic}
	Suppose that a sequence $f \colon \NN \to \RR$ is weakly periodic and generalised polynomial. Then there exists a periodic function $p$ and a set $Z \subset \NN$ of (upper Banach) density zero such that $f(n) = p(n)$ for $n \in \NN \setminus Z$. 
\end{alphatheorem}

To obtain stronger bounds on the size of the exceptional set $Z$, we need to restrict to automatic sequences and exploit some of their finer properties. Towards this end we develop a structure theory for sparse automatic sequences, i.e.\ those which take non-zero values on a set of integers of zero Banach density. In particular, we show that the set where a sparse automatic sequence takes non-zero values is either extremely small or combinatorially rich (see Theorem \ref{thm:Structure-Auto}). Conversely, we can show that sparse genearlised polynomials must be free of similar combinatorial structures. As a consequence, we prove the following.

\begin{alphatheorem}\label{thm:main-optimized}
	Suppose that a sequence $f \colon \NN \to \RR$ is automatic and generalised polynomial. Then there exists a periodic function $p$ and a set $Z \subset \NN$ such that $f(n) = p(n)$ for $n \in \NN \setminus Z$ and 
$$ \sup_{M} \abs{ Z \cap [M,M+N)} = O\bra{ \log^{C} N}$$	
	as $N \to \infty$ for a certain constant $C$.
\end{alphatheorem}

In fact, we obtain a much more precise structural description of the exceptional set $Z$. See Theorem \ref{thm:main-B2} for details.

The most technically difficult ingredient in our proof of Theorem \ref{thm:main-optimized} has to do with combinatorial structures in sparse in generalised polynomials. Using the techniques developed by Bergelson and Leibman, this reduces to questions concerning set of times an orbit on a nilmanifold hits a semialgebraic set (see Theorem \ref{thm:S:main-1} for the exact statements, and Section \ref{sec:DEF} for terminology). As an interesting by-product, we show that the fractional parts of an equidistributed polynomial sequence on a nilmanifold $G/\Gamma$ (with $G$ connected) cannot visit the projected image of a proper subgroup of $G$ infinitely many times (Proposition \ref{lem:fp-closure}); this result to the best of our knowledge is not present in the literature.

While Theorem \ref{thm:main-optimized} does not resolve Conjecture \ref{conjecture:main}, our proof thereof greatly restricts the number of possible counterexamples. In fact, in order to prove  Conjecture \ref{conjecture:main}, it would suffice to prove that the characteristic sequence of powers of an integer $k\geq 2$ given by $$g_k(n) =
	\begin{cases}
		1, & \text{ if } n = k^t \text{ for some } t \geq 0;\\
		0, & \text{ otherwise}
	\end{cases}
	$$ is not generalized polynomial.

\begin{alphatheorem}\label{thm:main-dichotomy}
	Let $k \geq 2$ be an integer. Then one of the following statements holds:
\begin{enumerate}
	\item All sequences which are simultaneously $k$-automatic and generalised polynomial are ultimately periodic.
	\item The characteristic sequence $g_k$ of the powers of $k$ is generalised polynomial.
\end{enumerate}
\end{alphatheorem}

This result prompts a more general question: For which exponential sequences $\floor{ \lambda^n + 1/2 }$ is the set of their values automatic? An example of such $\lambda$ is the golden ratio which corresponds to the Fibonacci numbers (Example \ref{ex:lin-rec-is-gp}). More subtle examples are provided by some Pisot numbers of degree $3$ (Theorem \ref{LAcrime}). We revisit this issue in Question \ref{Q:Pisot}.

\subsection*{Acknowledgements.} 

The authors thank Ben Green for much useful advice during the work on this project, Vitaly Bergelson for valuable comments on the distribution of generalised polynomials, and Jean-Paul Allouche for comments on related results on automatic sequences.



Thanks also go to Sean Eberhard, Dominik Kwietniak, Freddie Manners, Rudi Mrazovi\'{c}, Przemek Mazur, Sofia Lindqvist, and Aled Walker for many informal discussions.

This research was partially supported by the Polish National Science Centre (NCN) under grant no. DEC-2012/07/E/ST1/00185. JK also acknowledges the generous support from the Clarendon Fund and SJC Kendrew Fund for his doctoral studies.

Finally, we would like to express our gratitude to the organizers of the  conference \emph{New developments around x2 x3 conjecture and other classical problems in Ergodic Theory} in Cieplice, Poland in May 2016 where we began our project.

\section{Definitions} \label{sec:DEF}

\subsection*{Notations and generalities}

We denote the sets of positive integers and of nonnegative integers by $\N=\{1,2,\ldots\}$ and $\N_0=\{0,1,\ldots\}$. We denote by $[N]$ the set $[N]=\{0,1,\ldots, N-1\}.$ We use the Iverson's convention: whenever $p$ is any sentence, we denote by $\ifbra{p}$ its logical value ($1$ if $p$ is true and $0$ otherwise). We denote the number of elements in a finite set $A$ by $|A|$.

For a real number $r$, we denote by $\lfloor r \rfloor$ its integer part,  by $\{r\}=r-[r]$ its fractional part, by $\lbr r \rbr = \lfloor r+1/2\rfloor$ the nearest integer to $r$, and by $\fpa{r} = |r-\lbr r\rbr|$ the distance from $r$ to the nearest integer.

We use some standard asymptotic notation. Let $f$ and $g$ be two functions defined for sufficiently large integers. We say that $f=O(g)$ or $f \ll g$ if there exists $c>0$ such that $|f(n)|\leq c \abs{g(n)}$ for sufficiently large $n$. We say that $f=o(g)$ if for every $c >0$  we have  $|f(n)|\leq c |g(n)|$ for sufficiently large $n$. Finally, we say that $f=\Theta(g)$ if $f = O(g)$ and $g = O(f)$. 

For a subset $E\subset \N_0$, we say that $E$ has \emph{natural density} $d(A)$ if $$\lim_{N\to \infty} \frac{|E\cap [N]|}{N} = d(A).$$ 
We say that $E$ has \emph{upper Banach density} $d^*(A)$ if
 $$\limsup_{N\to \infty} \max_{M} \frac{|E\cap [M,M+N) |}{N} = d^*(A).$$

\comment{Actually, I think we only need upper Banach density. That's how I changed it above and also in the ergodic theorem. I hope it's fine.}

We now formally define generalised polynomials. 

\begin{definition}[Generalised polynomial]
	The family $\GP$ of \emph{generalised polynomials} is the smallest set of functions $\ZZ \to \RR$ containing the polynomial maps and closed under addition, multiplication, and the operation of taking the integer part. A set $E\subset \Z$ is called \emph{generalised polynomial} if its characteristic function given by $f(n)=\ifbra{n\in E}$ is a generalised polynomial.
\end{definition}

An example of a general polynomial would therefore be a function $f$ given by the formula $f(n)=2+\sqrt{2}\lfloor \sqrt{3}n^2+1/7\rfloor^2+n\lfloor n^3+\pi\rfloor$. 

\subsection*{Automatic sequences}

Whenever $A$ is a (finite) set, we denote by $A^*$ the free monoid with basis $A$. It consists of finite words in $A$, including the empty word $\epsilon$, with the operation of concatenation. We denote the contatenation of two words $v,w\in A^*$ by $vw$ and we denote the length of a word $w\in A^*$ by $|w|$. In particular, $|\epsilon|=0$. We say that a word $v\in A^*$ is a factor of a word $w \in A^*$ if there exist words $u,u' \in A^*$ such that $w=uvu'$. We denote by $w^R\in A^*$ the reversal of the word $w\in A^*$ (the word in which the elements of $A$ are written in the opposite order).

Let $k \geq 2$ be an integer and denote by $\Sigma_k=\{0,1,\ldots, k-1\}$ the set of digits in base $k$. For $w\in \Sigma_k^*$, we denote by $[w]_k$ the integer whose expansion in base $k$ is $w$, i.e. if $w=v_l v_{l-1}\cdots v_1 v_0$, $v_i \in\Sigma_k$, then $[w]_k=\sum_{i=0}^l v_ik^i$. Similarly, for an integer $n\geq 0$, we write $(n)_k\in \Sigma_k^*$ for the base $k$ representation of $n$ (without an initial zero). (In particular, $(0)_k = \epsilon$.)

The class of automatic sequences consists, informally speaking, of finite-valued sequences $(a_n)_{n\geq 0}$ whose values $a_n$ are obtained via a finite procedure from the digits of base $k$ expansion of an integer $n$. 

The most famous example of an automatic sequence is arguably the Thue-Morse sequence, first discovered by Prouhet in 1851. Let $s_2(n)$ denote the sum of digits in base 2 expansion of an integer $n$. Then the Thue-Morse sequence $(t_n)_{n\geq 0}$ is given by $t_n=1$ if $s_2(n)$ is odd and $t_n=0$ if $s_2(n)$ is even.

We will introduce the basic properties of automatic sequences. For more information, we refer the reader to the canonical book of Allouche and Shallit \cite{AS}. To formally introduce the notion of automatic sequences, we begin by finite automata.

\begin{definition}\label{automdef1} A finite $k$-\emph{automaton with output} (which we will just call a $k$-automaton) $\mathcal{A}=(S,s_{\bullet},\delta,\Omega,\tau)$ consists of the following data: \begin{enumerate} 
\item a finite set of states $S$;
\item an initial state $s_{\bullet}\in S$;
\item a transition map $\delta \colon S \times \Sigma_k \to S$;
\item an output set $\Omega$;
\item an output map $\tau \colon S \to \Omega$.
\end{enumerate}
We extend the map $\delta$ to a map $\tilde{\delta}\colon S \times \Sigma_k^* \to S$ by the recurrence formula $$\tilde{\delta}(s,\epsilon)=s, \quad \tilde{\delta}(s,wv)=\delta(\tilde{\delta}(s,w),v), \quad s\in S, w\in \Sigma_k^*, v\in\Sigma_k.$$ \end{definition}

We call a sequence $k$-\emph{automatic} if it can be produced by a $k$-automaton in the following manner: one starts at the initial state of the automaton, follows the digits of the base $k$ expansion of an integer $n$, and then uses the output function to print the $n$-th term of the sequence. This is stated more precisely in the following definition.

\begin{definition}\label{automdef2} A sequence $(a_n)_{n\geq 0}$ with values in a finite set $\Omega$ is $k$-\emph{automatic} if there exists a $k$-automaton $\mathcal{A}=(S,s_{\bullet},\delta,\Omega,\tau)$ such that $a_n=\tau\bra{\tilde{\delta}(s_{\bullet},(n)_k)}$. We call a set $E$ of nonnegative integers \emph{automatic} if the characteristic sequence $(a_n)$ of $E$ given by $a_n=\ifbra{n\in E}$ is automatic.\end{definition}

The values of the Thue-Morse sequence are  given by the  $2$-automaton 
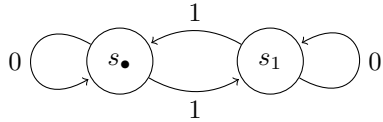
\begin{center}
\begin{tikzpicture}[shorten >=1pt,node distance=2cm, on grid, auto] 
   \node[state] (s_0)   {$s_{\bullet}$}; 
   \node[state] (s_1) [right=of s_0] {$s_1$}; 
  \tikzstyle{loop}=[min distance=6mm,in=210,out=150,looseness=7]
  
    \path[->] 
    
    (s_0) edge [loop left] node {0} (s_0)
          edge [bend right] node [below]  {1} (s_1);
          
 \tikzstyle{loop}=[min distance=6mm,in=30,out=-30,looseness=7]
 \path[->]
    (s_1) edge [bend right] node [above]  {1} (s_0)
          edge [loop right] node  {0} (s_1);
\end{tikzpicture}

\end{center}
with nodes depicting the states of the automaton, edges describing the transition map and with $\tau(s_{\bullet})=0$ and $\tau(s_1)=1$. Thus, the Thue-Morse sequence is $2$-automatic.

In the definition above,  the automaton reads the digits starting with the most significant one. In fact, we might equally well demand that the digits be read starting with the least significant digit or that the automaton produces the correct answer even if the input contains some leading zeroes. Neither of these modifications changes the notion of automatic sequences \cite[Theorem 5.2.3]{AS} (though of course we would need to use a different automaton to produce a given automatic sequence).

There is a number of equivalent definitions of automatic sequences connecting them to different branches of mathematics (stated for example in terms of algebraic power series over finite fields or letter-to-letter  projections of fixed points of uniform morphisms of free monoids). We will need one such definition that has a combinatorial flavour and is expressed in terms of the $k$-kernel.

\begin{definition}\label{automdef3} The $k$-\emph{kernel} $N_k((a_n))$ of a sequence $(a_n)_{n\geq 0}$ is the set of its subsequences of the form $$N_k((a_n))=\{(a_{k^l n+r})_{n\geq 0} \mid l\geq 0, 0\leq r < k^l\}.$$
\end{definition}

Since for the Thue-Morse sequence we have relations $t_{2n}=t_n$, $t_{2n+1}=1-t_n$, one easily sees that the $2$-kernel $N_2((t_n))$ consists of only two sequences $N_2((t_n))=\{t_n, 1-t_n\}$. This gives another argument for its $2$-automaticity.

\begin{proposition}\label{automthm1}\cite[Theorem 6.6.2]{AS} Let $(a_n)_{n\geq 0}$ be a sequence. Then the following conditions are equivalent: \begin{enumerate} 
\item The sequence $(a_n)$ is $k$-automatic. 
\item The $k$-kernel $N_k((a_n))$ is finite.\end{enumerate}\end{proposition}

An automatic sequence by definition takes only finitely many values. In 1992 Allouche and Shalit \cite{AlloucheShallit-2003} generalized automatic sequences to a wider class of  $k$-regular sequences  that are allowed to take values in a possible infinite set. The definition of regular sequences is stated in terms of the $k$-kernel. For simplicity, we state the definition for sequences taking integer values, though it could also be introduced for sequences taking values in a (noetherian) ring.

\begin{definition}\label{automdef4} Let $(a_n)_{n\geq 0}$ be a sequence of integers. We say that the sequence $(a_n)$ is $k$-\emph{regular} if its $k$-kernel $N_k((a_n))$ spans a finitely generated abelian subgroup of $\Z^{\N_0}$.
\end{definition}

For example, the following sequences are easily seen to be $2$-regular: $(t_n)_{n\geq 0}$, $(n^3+5)_{n\geq 0}$, $(s_2(n))_{n\geq 0}$. (The corresponding subgroups spanned by the $2$-kernel have rank $2$, $4$, and $2$. In the case of $t=(t_n)_{n\geq 0}$, the subgroup spanned by the $2$-kernel is free abelian with basis consisting of $t$ and the constant sequence $(1)_{n\geq 0}$.) In fact, every $k$-automatic (integer-valued) sequence is obviously $k$-regular, and the following converse result holds.

\begin{proposition}\label{automthm2}\cite[Theorem 16.1.5]{AS}
Let  $(a_n)_{n\geq 0}$ be a sequence of integers. Then the following conditions are equivalent: \begin{enumerate} 
\item The sequence $(a_n)$ is $k$-automatic. 
\item The sequence $(a_n)$ is $k$-regular and takes only finitely many values.\end{enumerate}\end{proposition}

\begin{corollary}\label{automthm3}\cite[Corollary 16.1.6]{AS}
Let  $(a_n)_{n\geq 0}$ be a sequence of integers that is $k$-regular and let $m\geq 1$ be an integer. Then the sequence $a_n \bmod m$ is $k$-automatic.

\end{corollary}

A convenient tool for ruling out that a given sequence is automatic is provided by the Pumping Lemma.
\begin{lemma}\label{lem:pumping}\cite[Lemma 4.2.1]{AlloucheShallit-book} 	
	 Let $(a_n)_{n\geq 0}$ be a $k$-automatic sequence. Then, there exists a constant $N$ such that for any $w \in \Sigma^*_k$ with $\abs{w} \geq N$ and any $L \leq \abs{w} - N$ there exist $u_0, u_1, v \in \Sigma^*_k$ with $w = u_0 v u_1$ such that $L \leq \abs{u_0} \leq L + N - \abs{v}$, and $a_n$ takes the same value for all $n \in \set{ [u_0 v^t u_1]_k}{t \in \NN_0}$. 
\end{lemma}
\comment{This is not precisely the same as the formulation in AS, but close enough. The proof is trivial in any case, and derivation of the above from the lemma in AS is even more trivial. I suggest not to add anything more on that in the paper, but I figured I should mention that. --JK
I'm fine with that --JB}

The final issue that we need to discuss is the dependence of the notion of $k$-automaticity on the base $k$. While the Thue-Morse sequence is $2$-regular, and is also easily seen to be $4$-regular, it is not $3$-regular. This follows from the celebrated result of Cobham \cite{Cobham-1969}. We say that two integers $k,l\geq 2$ are \emph{multiplicatively independent} if they are not both powers of the same integer (equivalently, $\log k/\log l \notin \Q$).

\begin{theorem}\label{automthm4}\cite[Theorem 11.2.2]{AS} Let $(a_n)_{n\geq 0}$ be a sequence with values in a finite set $\Omega$. Assume that the sequence $(a_n)$ is simultaneously $k$-automatic and $l$-automatic with respect to two multiplicatively independent integers $k,l\geq 2$. Then $(a_n)$ is eventually periodic.\end{theorem}

We will have no use for Cobham's Theorem. We will however use the following much easier related result.

\begin{theorem}\label{automthm5}\cite[Theorem 6.6.4]{AS} Let $(a_n)_{n\geq 0}$ be a sequences with values in a finite set $\Omega$. Let $k,l\geq 2$ be two multiplicatively \emph{dependent} integers. Then the sequence $(a_n)$ is $k$-automatic if and only if it is $l$-automatic.\end{theorem}

\subsection*{Dynamical systems}
An (invertible, topological) dynamical system is given by a compact metric space $X$ and a continuous homeomorphism $T\colon X\to X$. We say that $X$ is transitive if there exists $x\in X$ whose orbit $\{T^n x \mid n\in \Z\}$ is dense in $X$. We say that $X$ is minimal if for every point $x\in X$ the orbit $\{T^n x \mid n\in \Z\}$ is dense in $X$. (Equivalently, the only closed subsets $Y\subset X$ such that $T(Y)= Y$ are $Y=X$ or $Y=\emptyset$.) We say that $X$ is totally minimal if the system $(X,T^n)$ is minimal for all $n\geq 1$. 

Let $(X,T)$ be a dynamical system. We say that a Borel measure on $X$ is invariant if for every Borel subset $A\subset X$ we have $\mu(T^{-1}(A))=\mu(A)$. By the Krylov-Bogoliubov theorem, each dynamical system has at least one invariant measure. We say that a dynamical system in \emph{uniquely ergodic} if it has exactly one invariant measure. We say that an invariant measure $\mu$ on $X$ is ergodic if for every $A\subset X$, $T^{-1}(A)=A$ implies $\mu(A)\in\{0,1\}$. 

If $(X,T)$ is minimal, $x \in X$ and $U \subset X$ is open, then the set $\set{n \in \ZZ}{T^n x \in U}$ is syndetic, i.e.\ has bounded gaps \cite[Thm.\ 1.15]{Furstenberg1981}.

We will frequently use the fact that a minimal system $(X,T)$ with $X$ connected is totally minimal. (Otherwise, there exists $d \in \NN$ and a closed subset $\emptyset \neq Y \subset X$ such that $T^d(Y) = Y$. By Kuratowski-Zorn Lemma there exists a minimal $Y$ with this property. The sets $Y,T(Y),T^{2}(Y),\dots,T^{d-1}(Y)$ are disjoint, closed, and their union is dense in $X$. This contradicts connectedness.)

\begin{theorem}[Ergodic theorem for uniquely ergodic system, \cite{EinsiedlerWard} Thm 4.10]\label{thm:ergo-thm-uniform}
	Let $(X,T)$ be a dynamical system with unique ergodic measure $\mu$. Then, for any $f \in C(X)$, 
	$$
		\frac{1}{N} \sum_{n = 0}^{N - 1} f(T^n x) \to \int_X f d\mu,
	$$
	as $N \to \infty$, uniformly in $x \in X$.
\end{theorem} 

By a standard argument that bounds the characteristic function of a set from above and from below by continuous functions, we obtain the following result, frequently used in the sequel.
\begin{corollary}\label{cor:density-uniform}
	Let $(X,T)$ be a uniquely ergodic dynamical system with the invariant measure $\mu$. Then, for any $x \in X$ and any $S \subset X$ with $\mu(\partial S) = 0$, the set $E = \set{n \in \NN_0}{T^n x \in S}$  has (upper) Banach density $\mu(S)$.
\end{corollary}
In fact, in this case the limit superior in the definition of (upper) Banach density can be replaced by a limit.

\subsection*{Generalized polynomials and nilmanifolds}\label{sec:DEF-GP}

A nilmanifold is a homogenous space $X=G/\mathord \Gamma$, where $G$ is a nilpotent Lie group and $\Gamma$ is a discrete cocompact subgroup, together with the action of $G$ on $X$ via left translations $G\times X \to X$. We do not assume that $G$ is connected, however in all the cases that we will consider the connected component $G^{\circ}$ of $G$ will be \emph{simply connected}. We assume this henceforth, since it will substantially simplify the discussion.

We begin by recalling a few basic facts concerning nilpotent Lie groups. We follow the presentation in \cite{BergelsonLeibman2007} and \cite{Malcev1951}, which the reader should consult for the proofs. Let $G$ be a connected simply connected nilpotent Lie group. For each $g\in G$, there is a unique one-parameter subgroup $\{g^t\}_{t\in\R}$ of $G$, i.e. a continuous homomorphism $\R \to G$, $t\mapsto g^t$ with $g^1=g$. Consider the lower central series 
$$ G_0 = G_1 \supset G_2 \supset \ldots\supset G_d\supset G_{d+1}=1$$ 
given by $G = G_0 = G_1$ and $G_{i+1}=[G,G_i]$,
 $1\leq i \leq d$. The subgroups $G_i$ are closed and $G_i/\mathord G_{i+1}$ are finite dimensional $\R$-vector spaces with the action of $\R$ given by one-parameter subgroups. 
 
 Let $l_i=\dim_{\R} G_i/\mathord G_{i+1}$, $k_i=\sum_{j=1}^i l_j$, $1\leq i \leq d$. Then $G$ has a \emph{\Malcev\ basis}, i.e. elements $e_1,\ldots,e_k\in G$, $k=k_d$, such that $e_{k_i+1},\ldots,e_{k_{i+1}}\in G_i$ and their images in $G_i/\mathord G_{i+1}$ constitute a basis of $G_i/\mathord G_{i+1}$ as a $\R$-vector space.

It follows easily that any element $g\in G$ can be written uniquely in the form $$g=e_1^{t_1}\cdots e_k^{t_k}, \quad t_i\in \R.$$ Furthermore, it is possible to choose a \Malcev\ basis to be compatible with $\Gamma$, i.e. so that $g=e_1^{t_1}\cdots e_k^{t_k}$ is in $\Gamma$ if and only if $t_1,\ldots,t_k \in\Z$. We will always assume that our \Malcev\ bases are compatible with $\Gamma$. A choice of a \Malcev\ basis determines a diffeomorphism $$ \tilde \tau \colon  G \to \R^k, \quad  e_1^{t_1}\cdots e_k^{t_k} \mapsto (t_1,\ldots, t_k).$$
Under this identification, multiplication on $G$ is given by polynomial formulas.

Let $Q=\{(t_1, \ldots,t_k)\in \R^k \mid 0\leq t_i < 1, 1\leq i\leq k\}$ be the unit cube. The preimage $D = \tau^{-1}(Q) \subset G$ is the \emph{fundamental domain}. For any $g\in G$, there is a unique choice of elements $[g]\in \Gamma, \{g \}\in D$, called the \emph{integral part} and the \emph{fractional part} of $g$ respecitvely, such that $\fp{g} \ip{g} = g$. (Note that the elements $\{g\}$ and $[g]$ depend on the choice of the \Malcev\ basis.) Hence, the map $\tilde \tau$ factors to a continuous bijection
$$ \tau \colon  G/\Gamma \to [0,1)^k,$$
which will play a crucial r\^ole.

The choice of \Malcev\ coordinates also induces a natural choice of metric on $G$ and $G/\Gamma$ (cf.\ \cite[Def. 2.2]{GreenTao2012}). 

A \emph{horizontal character} is a non-trivial morphism of groups $\eta \colon G \to \RR/\ZZ$ with $\Gamma \subset \eta^{-1}(0)$, and is necessarily of the form $\eta(g) = \sum_{i=1}^{l_1} a_i \tau_i(g\Gamma)$, where $a_i \in \ZZ$, not all $0$. The \emph{norm} of $\eta$, denoted $\norm{\eta}$, is by definition the Euclidean norm $\norm{ (a_i) }_2$. 

A set $A \subset \R^k$ is called semialgebraic if it is a finite union of subsets of $\R^k$ given by a system of finitely many polynomial equalities and inequalities; in particular $A$ is Borel and of zero Lebesgue measure. A map $p\colon A \to \R^k$ defined on a semialgebraic set $A\subset \R^k$ is called  \emph{piecewise polynomial} if there is a decomposition $A=A_1\cup \ldots \cup A_s$ into a finite union of semialgebraic sets such that $p|_{A_i} \colon A_i \to \R^l$ is a polynomial mapping restricted to $A_i$. We call a map $p\colon X=G/\mathord \Gamma \to \R^l$ \emph{piecewise polynomial} if it takes the form $p = q \circ \tau$ where $q \colon Q \to \R^l$ is a piecewise polynomial map. While the map $\tau\colon D \to X$ depends on the choice of the \Malcev\ basis, the concept of a piecewise polynomial map on $X$ does not. Note that if $p\colon X \to \R^l$ is a piecewise polynomial map and $g\in G$, then so is $q\colon X \to \R^l$ given by $q(h\Gamma)=p(gh\Gamma)$.

We now extend the above definitions to not necessarily connected simply connected Lie groups $G$. A \Malcev\ basis of $G$ is simply a \Malcev\ basis of its connected component $G^{\circ}$. Similarly, we define piecewise polynomial maps $p\colon X \to \R^l$ in terms of their restrictions to the connected components of $X$ by demanding that for each $g\in G$ the map $G^{\circ}/\mathord (\Gamma \cap G^{\circ}) \to \R^l$, $h(\Gamma \cap G^{\circ})\mapsto p(gh\Gamma)$. (It is enough to verify this condition only for a single $g$ in each class $g\Gamma \in G/\mathord G_0$.)

There is a natural way to introduce dynamics on a nilmanifold Any $g \in G$ acts on $X$ by left translation $T_g(h\Gamma) = gh\Gamma$. There is a unique Haar measure $\mu_{X}$ on $G/\Gamma$ invariant under the left translations. For any $g\in G$, we obtain a dynamical system $(X,T_g)$. By a result of Parry \cite{Parry1969}, the conditions of minimality, unique ergodicity, and ergodicity with respect to the measure $\mu_X$ are equivalent. (In fact, all this conditions can be verified on the maximal torus $G/G_2\Gamma \simeq \TT^{l_1}$.)

Generalised polynomials and dynamics on nilmanifolds are intimately related. This is explained in particular by the main result of \cite[Theorem A]{BergelsonLeibman2007}.

\begin{theorem}[Bergelson-Leibman]\label{BLnilgenpolythm}\label{thm:BergelsonLeibman}
 \begin{enumerate} 
\item If $X$ is a nilmanifold, $g\in G$ acts on $X$ by left translations, $p\colon X \to \R$ is a piecewise polynomial map, and $x\in X$, then $u(n)=p(g^n x), n\in \Z$ is a bounded generalised polynomial.
\item If $u\colon \Z \to \R$ is a bounded generalised polynomial, then there exists a nilmanifold, $g\in G$ acting on $X$ by left translations in such a way that the action is ergodic, a piecewise polynomial map $p\colon X \to \R$, and $x\in X$ such that $u(n)=p(g^n x)$, $n\in \Z$.\end{enumerate}\end{theorem}

We finish by defining \emph{polynomial sequences} with values in nilpotent groups. Strictly speaking, polynomial sequences are defined relative to a filtration; we only work with polynomial sequences with respect to the lower central series.

\begin{definition} Let $G$ be a nilpotent group and let $G= G_0 = G_1 \supset G_2 \supset \ldots \supset G_d \supset G_{d+1}=1$ be its lower central series. A sequence $g\colon \ZZ \to G$ is polynomial if it takes the form
	$$
		g(n) = g_0^{\binom{n}{0}} g_1^{\binom{n}{1}} g_2^{\binom{n}{2}} \dots g_d^{\binom{n}{d}}$$ for some $g_i \in G_i,\ 0 \leq i \leq d$.
\end{definition}

It has been proven by Lazard and Leibman that polynomial sequences form a group with termwise multiplication. For more details, see \cite{Lazard-1954}, \cite{Leibman1998}, \cite{Leibman2002}. This is one of the many reasons why it is often more convenient to work with polynomial sequences, even if one is ultimately interested in linear orbits.

Distribution properties of polynomial sequences have been extensively studied. In the quantitative setting we have the following. We will also later need a quantitative variant due to Green and Tao (cf.\ Theorem \ref{thm:GreenTao}).

\begin{theorem}[Leibman]
	Let $G/\Gamma$ be a nilmanifold, and let $g \colon \ZZ \to G$ be a polynomial sequence. Then, either $(g(n)\Gamma)_{n \geq 0}$ is equidistributed in $X$ (with respect to $\mu_X$), or there exists a horizontal character $\eta$ such that $\eta \circ g$ is constant.
\end{theorem}

\subsection*{\IP-sets and \IP-convergence}

We denote by $\cF$ the set of all finite nonempty subsets $\a \subset \NN$ and we put $\cFe = \cF \cup \{\emptyset\}$. For sets $\alpha,\beta \in \cF$, we write $\alpha < \beta$ if $a<b$ for all $a\in \alpha, b\in \beta$. We will extensively use sequences $(n_{\alpha})_{\alpha \in \cF}$ indexed by finite sets of natural numbers. In fact, for a sequence $(n_i)_{i\in \N}$ of integers, we will frequently use the associated sequence $(n_{\alpha})_{\alpha \in \cF}$ given by $n_{\alpha}=\sum_{j\in \alpha} n_j$ and denote its set of finite sums by $\FS(n_i) = \set{ n_\a }{ \a \in \cF}.$

The following notion is widely used and well-studied.

\begin{definition}[\IP-set]\label{def:IP-set} A set $E\subset \N$ is called an \emph{\IP-set} if it contains a set of the form  $\FS(n_i)$ for some natural numbers $(n_i)_{i\in \N}$. Similarly, $E \subset \NN$ is an \IPP-set if it contains $E_0 + a$ for some $a \in \ZZ$ and \IP-set $E_0$.  

A set $E\subset \N$ is called an \emph{\IPd-set} if it intersects nontrivially any \IP-set.
\end{definition}

The class of \IP-sets is partition regular, i.e.\ whenever an \IP-set $E$ is written as a union of finitely many subsets $E=E_1\cup E_2 \cup \ldots \cup E_r$, at least one of the sets $E_i$ is an \IP-set. This is the statement of Hindman's Theorem (see \cite{Hindman1974}, \cite{Bergelson2010d}, or \cite{Baumgartner-1974}). It follows that an \IPd-set is an \IP-set, an intersection of two \IPd-sets is an \IPd-set, and an intersection of an \IPd-set and an \IP-set is an \IP-set. (For the first and the third statement, let $A$ be an \IPd-set and let $B$ be an \IP-set. Apply Hindman's theorem to the partitions $\N=A \cup (\N\setminus A)$ and  $B=(B\cap A) \cup (B\setminus A)$. The second statement then easily follows.)

\IP-sets occur naturally in dynamics. Let $(X,T)$ be a minimal topological dynamical system. Then for every $x\in X$ and every neighbourhood $U$ of $x$, the set of return times $\{n\in \N \mid T^n(x)\in U\}$ is an \IPd-set (see e.g.\ \cite[Lemma 9.10]{Furstenberg1981}).

A related concept is that of an \IP-ring. A family $\mathcal{G}\subset \cF$ of subsets of $\N$ is called an \emph{\IP-ring} if there exists a sequence $\beta=(\beta_i)_{i\in \N}\subset \cF$ with $\beta_1<\beta_2<\ldots$ and such that $$\mathcal{G}=\Big\{\bigcup_{i \in \gamma} \beta_i \ \Big| \ \gamma \in \cF \Big\}.$$ An equivalent version of Hindman's theorem says that whenever  an \IP-ring $\mathcal{G}$ is written as a union of finitely many subfamilies $\mathcal{G}=\mathcal{G}_1\cup \mathcal{G}_2 \cup \ldots \cup \mathcal{G}_r$, at least one of the subfamilies $\mathcal{G}_i$ contains an \IP-ring (cf.\ \cite{Bergelson2010d}, \cite{Baumgartner-1974}.)

\newcommand{\iplim}{\mathrm{IP-}\!\lim}

There is a natural notion of convergence for sequences indexed by $\cF$, or more generally an \IP-ring. Let $X$ be a topological space and let $(x_\alpha)_{\alpha \in \cF}$ be a sequence of points of $X$. We say that the sequence $(x_\alpha)_{\alpha \in \cF}$ converges to a limit $x\in X$ along an \IP-ring $\mathcal{G}$ and we write $$\iplim_{\a \in \mathcal{G}} x_{\alpha}  = x$$ if for any neighbourhood $U$ of $x$ there exists $\alpha_0\in \cF$ such that for all $\alpha \in \mathcal{G}$ with $\alpha>\alpha_0$ we have $x_{\alpha}\in U$. Limits along \IP-rings are closely related to limits along idempotent ultrafilters in $\N$.

Let $X$ be a compact topological space let $(x_\alpha)_{\alpha \in \mathcal{F}}$ be a sequence of points of $X$. It is a corollary of Hindman's theorem and a diagonal argument (\cite[Theorem 8.14]{Furstenberg1981} or \cite[Theorem 1.3]{FurstenbergKatznelson-1985} \comment{CHECK SECOND REFERENCE}) that there exists an \IP-ring $\mathcal{G}$ and a point $x\in X$ such that $\iplim_{\a \in \mathcal{G}} x_{\alpha}  = x$.


\comment{A general, very pedantic comment: When writing i.e.\ should write $\text{i.e.}\backslash$ or $\text{i.e.}\{\}$, because otherwise the spacing is wrong.

See e.g.\ http://tex.stackexchange.com/questions/2229/is-a-period-after-an-abbreviation-the-same-as-an-end-of-sentence-period}

\section{Density 1 results}\label{sec:MAIN}
\comment{Modulo the ``comments'', I am happy with what this section looks like now. Feel free to treat this as a final version, in particular suggest changes, edit, etc.}

\comment{Section name remains to be discussed}

\subsection*{Polynomial sequences}
Our first purpose in this section is to prove Theorem \ref{thm:main-sortof}. Recall that we aim to show that the sequence $n \mapsto \floor{ p(n) }$ is not regular, where $p \in \RR[x]$ has at least one irrational coefficient other than the constant term. We will show more, namely that the sequence $m \mapsto \floor{p(n)} \bmod{m}$ is not automatic for $m \geq 2$. 

In fact, we will only need to work with the weaker property of weak periodicity, defined in the introduction. 

\begin{lemma}\label{lem:auto=>weak-per}
	Any automatic sequence is weakly periodic.
\end{lemma}
\begin{proof}
	Let $f$ be a $k$-automatic sequence. Since restriction to an arithmetic progression of a $k$-automatic sequence is again $k$-automatic, it will suffice to find $q \in \NN$, $r,r' \in \NN_0$ with $r \neq r'$ such that $f(qn + r) = f(qn+r')$.
\comment{We no longer insist that $q > r,r'$; correct the definition earlier on. Probably it's more natural this way.}	
	
	The $k$-kernel $N_k(f)$ of $f$, consisting of the functions $f(k^t n + r)$ for $0 \leq r < k^t$, is finite. Pick $t$ sufficiently large that $k^t > \abs{N_k(f)}$. By the pidgeonhole principle, there exist $r \neq r'$ such that $f(k^t n + r) = f(k^t n + r')$. 
\end{proof}

\begin{proposition}[Polynomial sequences are not weakly periodic]\label{prop:poly=>not-wp}
Let $p(x) \in \RR[x]$ be a polynomial, and let an $m \geq 2$ be an integer. Then, the sequence $n \mapsto \floor{p(n)}\bmod{m}$ is weakly periodic if and only if it is periodic. This happens precisely when all non-constant coefficients of $p(x)$ are rational.
\end{proposition}
\begin{proof}
	If all coefficients of $p(x)$ are rational, except possibly the constant term, then the sequence $n \mapsto \floor{p(n)}\bmod{m}$ is easily checked to be periodic, hence weakly periodic. 

	Suppose now that at least one non-constant coefficient of $p(x)$ is irrational. Replacing $p(x)$ with $p(h x + r)$ for multiplicatively large $h$ if need be, we may assume that the leading coefficient of $p(x)$ is irrational. We will prove marginally more than claimed, namely that for any $0 \leq r < m$, the sequence $f$ given by 
	\begin{equation}\label{eq:104}
		f(n) = 
		\begin{cases}
			1 & \text{if } \floor{p(n)} \equiv r \pmod{m},\\
			0 & \text{otherwise,}
		\end{cases}
	\end{equation}
	fails to be $k$-automatic for any $k$. Fix $k$ and $r$, and suppose for the sake of contradiction that $f$ is $k$-automatic.
	
	It will be convenient to expand $\frac{1}{m} p(x) = \sum_{i = 0}^d a_i \binom{x}{i}$, where $d = \deg p$ and $\binom{x}{i} = \frac{1}{i!} x(x-1)(x-2) \dots (x-i+1)$. Note that $a_d \in \RR \setminus \QQ$ and
	\begin{equation}\label{eq:103}
		f(n) = 
		\begin{cases}
			1, & \text{if } \floor{ \frac{1}{m}p(n)} \bmod{1} \in \left[ \frac{r}{m}, \frac{r+1}{m}\right), \\
			0, & \text{otherwise.}
		\end{cases}
	\end{equation}
	
	We will represent $p(n)$ dynamically. Let $X = \TT^d = \RR^d/\ZZ^d$, and define the self-map $T \colon X \to X$ by 
	\begin{equation}\label{eq:100}
			(x_1,x_2,x_3,\dots,x_d) \mapsto (x_1 + a_d, x_2 + x_1 + a_{d-1}, 
		 \dots, x_d + x_{d-1} + a_1).
	\end{equation}
	\comment{We could also define $T$ by the simpler formula 
	$(x_1 + a_d, x_2 + x_1, x_3 + x_2, \dots, x_d + x_{d-1})$
	and use $z = (a_{d-1},a_{d-2},\dots,a_1,a_0)$. Not terribly important, but if you (=JB) want to change this, I'm open. --JK} 
	
	A standard computation shows that for $z = (0,0,\dots,0,a_0)$ we have
	\begin{equation}\label{eq:101}
		(T^n z)_j = z_j + \sum_{i \geq 1} a_{d-j+i} \binom{n}{i},
	\end{equation}
and in particular $(T^n z)_d = \frac{1}{m}p(n).$ Putting $A = \TT^{d-1} \times \left[ \frac{r}{m}, \frac{r+1}{m}\right)$ we thus find 
	\begin{equation}\label{eq:102}
		f(n) = 
		\begin{cases}
			1, & \text{if } T^n z \in A, \\
			0, & \text{otherwise.}
		\end{cases}
	\end{equation}

Because $f$ is weakly periodic, we may find $q$ and $r \neq r'$ such that $f(qn + r) = f(qn + r')$.
	The dynamical system $(X,T)$ is known to be totally minimal (see e.g.\ \cite[Section 4.4.3]{EinsiedlerWard}). In particular, for any point $y \in \cl A$ we may find a sequence $(n_i)_{\geq 1}$ such that $T^{q n_i + r} z \to y$ and $T^{q n_i + r} z \in A$. It follows that the the points $T^{q n_i + r'} z$ converge to $T^{r'-r}y$ and lie in $A$. Thus, $T^{r'-r} (\cl A) \subset \cl A$. In light of minimality, this is only possible if $\cl A = X	$ or $\cl A = \emptyset$ --- but this is absurd.
\end{proof}

\begin{corollary}\label{prop:poly=>not-auto}
	With notation as in Proposition \ref{prop:poly=>not-wp}, the sequence $n \mapsto \floor{p(n)} \bmod{m}$ is automatic if and only if it is periodic.
\end{corollary}
\begin{proof}
	Immediate from Proposition \ref{prop:poly=>not-wp} and Lemma \ref{lem:auto=>weak-per}.
\end{proof}

\begin{proof}[Proof of Theorem \ref{thm:main-sortof}]
	Suppose first that all non-constant coefficients of $p(n)$ are rational, and fix $k$. Let $h \in \NN$ be such that $h p(n)$ has integer coefficients, possibly except the constant term. Then $f_1(n) = \floor{ h p(n) }$ is an integer valued polynomial, hence $k$-regular ($N_k(f_1)$ is contained in the $\deg p + 1$ dimensional $\ZZ$-module consisting of $\ZZ$-valued polynomials of degree $\leq \deg p$). Also, $f_2(n) = \floor{ h p(n) } - h f(n) = \floor{h \fp{p(n)}}$ is periodic, hence $k$-automatic, hence $k$-regular. It follows that $f(n) = \frac{1}{h} \bra{ f_1(n) - f_2(n) }$ is rational.

	
	Conversely, suppose that $f(n)$ is regular. Then, by Theorem \ref{automthm2}, for any choice of $m \geq 2$, $f(n) \bmod{m}$ is automatic. Now, it follows from Corollary \ref{prop:poly=>not-auto} that all non-constant coefficients of $p(n)$ are rational. 
\end{proof}

\comment{The first part would be a bit cleaner if we did not insist that regular sequences are $\ZZ$-valued. Do we have a strong reason for that? Why not e.g.\ work with $\RR$- or even $\CC$-valued ones?}

\subsection*{Generalised polynomials}
Having dealt with the case of polynomial maps, we move on to a more general context. Our next goal is the proof of Theorem \ref{thm:main-weakly-periodic}. We begin with abstracting and generalising some of the key steps from the proof of Theorem \ref{thm:main-sortof}.

 Recall that a set of integers is \emph{thick} if it contains arbitrarily long segments, and \emph{syndetic} if it has bounded gaps; any thick set intersects any syndetic set. 

\begin{lemma}[Totally minimal sequences are not weakly periodic]\label{lem:tot-min=>not-auto}
	Let $(X,T)$ be a totally minimal dynamical system. Let $A \subset X$, be a set with $\cl A = \cl \inter A \neq \emptyset, X$ and let $z \in X$. Suppose that $f \colon \NN_0 \to \{ 0, 1\}$ is a sequence such the set of $n$ with $f(n) = \ifbra{T^n z \in A}$ is thick. Then $f$ is not weakly periodic. 
\end{lemma}
\begin{proof}
	Suppose for the sake of contradiction that $f$ were weakly periodic. In particular there are some $q \in \NN$, $r,r' \in \NN_0$ with $r \neq r'$ such that $f(qn + r) = f(qn + r')$. Put $d = r'-r$. 
	
	We will show that that $T^{d} (\cl A) \subset \cl A$; because $T$ is continuous and $\cl \inter A = \cl A$, it will suffice to prove that $T^{d} (\inter A)  \subset \cl A$. Once this is accomplished, the contradiction follows immediately, because $(X,T^{d})$ is minimal, while $\cl A \neq \emptyset,X$.

	Pick any $y \in \inter A$ and an open neighbourhood $T^{d}y \in V$; we aim to show that $V \cap A \neq \emptyset$. Put $U = T^{-d} V \cap \inter A \neq \emptyset$, and consider the set $S$ of those $n$ for which $T^{q n + r}z \in U$. Because $(X,T^{q})$ is minimal, $S$ is syndetic. Let $R_0$ be the set of those $n$ for which $f(n) = \ifbra{T^n z \in A}$ and put $R = (R_0 - r)/q$ and $R' = (R_0 - r')/q$. 
	
	Because $R_0$ is thick, $R \cap R'$ is thick, and because $S$ is syndetic, $S \cap R \cap R'$ is non-empty. Pick any $n \in S \cap R \cap R'$, and put $x = T^{qn + r} z$. Because $n \in S$, we have $x \in U$, and $T^d x \in V$. Since $n \in R$, we have $f(qn +r) = \ifbra{x \in A} = 1$, and hence also $f(qn + r') = 1$. Finally, because $n \in R'$, we have $1 = f(q n + r') = \ifbra{T^{d} x \in A}$, meaning that $T^{d} x \in V \cap A$. In particular, $ V \cap A \neq \emptyset$, which was our goal.
	\end{proof}

\begin{remark}
	Some mild topological restrictions on the target set $A$ are, of course, necessary in the above lemma. Note that any open, non-dense and non-empty subset of $X$ will satisfy the stated assumptions.
\end{remark}

The analogue of the representation of a polynomial sequence using a skew rotation on the torus in \eqref{eq:102} is provided by the Bergelson-Leibman Theorem \ref{thm:BergelsonLeibman}. In the special case of a generalised polynomial $g(n)$ taking values in a finite set $\Omega = \{c_1,\dots,c_r\}$, it asserts that there exists a minimal nilsystem $(X,T)$, a point $z \in X$, and a partition $X =  S_1 \cup S_2 \cup \dots \cup S_r$ into semialgebraic pieces, such that  $g(n) = c_j$ if and only if $T^n z \in S_j$.

We are now ready to state and prove the main result of this section, from which Theorem \ref{thm:main-weakly-periodic} easily follows.

\begin{theorem}\label{thm:gen-poly=>not-auto-dens-1}
	Let $g \colon \ZZ \to \RR$ be a generalised polynomial taking finitely many values, and let $f\colon \NN_0 \to \RR$ be a weakly periodic sequence which agrees with $g$ on a thick set $R \subset \NN_0$. Then, there exists a set $Z \subset R$ with $d^*(Z) = 0$, such that the common restriction of $f$ and $g$ to $R \setminus Z$ is periodic.	
\end{theorem}

\begin{proof}
	As a special case of Theorem \ref{thm:BergelsonLeibman}, we may find a minimal system $(X,T)$, $z \in X$, a partition $X = \bigcup_{j=1}^r S_j$ with $\inter \partial S_j = \emptyset$ for all $j$, and constant $c_j$, such that 
	\begin{equation}\label{eq:107}
	g(n) = \sum_{j=1}^r \ifbra{T^n z \in S_j} c_j.
	\end{equation}	
	
	If $X$ is disconnected, then it has finitely many connected components $X_1, \dots, X_m$, which are permuted by $T$. Hence, for $a = m!$ and any $0 \leq b < a$, the function $g'(n) = g(a n + b)$ can be represented as in \eqref{eq:107} using a connected system. Clearly, $f'(n) = f(an + b)$ is weakly periodic and agrees with $g'(n)$ on the thick set $R' = (R-b)/a$. Thus, it will suffice to prove the theorem under the additional assumption that $X$ is connected. In this case, $(X,T)$ is totally minimal. 
	\comment{Facts used: nilmanifolds have finitely many connected compontents (trivial); nothing more really.} \comment{I almost thought we had this one wrong for a moment. Note that while we may ask that $X$ be connected, we may not ask that $G$ be connected. Anyways, at this stage we may just as well forget about nilmanifold structure and only think about minimality.}
	
	We may write 
	\begin{equation}\label{eq:108}
g(n) = \sum_{i=j}^r \ifbra{ T^n z \in \inter S_j} + h(n),
	\end{equation}	
 where $h(n) = 0$ unless $T^n z \in \bigcup_{j=1}^r \partial S_j$. In particular (by Corollary \ref{cor:density-uniform}) the set $Z \subset \NN_0$ of $n$ with $h(n) \neq 0$ has upper Banach density $0$. Note that $R \setminus Z$ is then thick. 
	
	For $j \in [r]$, put $g_j'(n) = \ifbra{T^n z \in \inter S_j}$ and $f'_j(n) = \ifbra{f(n) = c_j}$. Then $g_j'(n) = f_j'(n)$ for $n \in R \setminus Z$. By Lemma \ref{lem:tot-min=>not-auto}, this is only possible if for each $j$, the set $\inter S_j$ is either empty or dense. Since $\mu( X \setminus \bigcup_{j=1}^r \inter S_j) = 0$, there is one $i$ such that $\inter S_i$ is dense, and $\inter S_j = \emptyset$ for $j \neq i$. Denoting by $Z'$ the set of $n \in R$ with $T^n z \in X \setminus \inter S_i$ we have $d^*(Z') = 0$ and $g(n) = c_i$ for $n \in R \setminus Z'$, as needed.	
\end{proof}
\comment{
A general remark on style: I totally disregard the notation introduced in conclusions of theorems. Hence, for instance here, I don't care that $Z$ already exists in the conclusion of the theorem, and feel free to use $Z$ to denote the set in the middle of the proof, and then use $Z'$ for the sought set. I am open to changing that if need be. -- JK
}

\begin{proof}[Proof of Theorem \ref{thm:main-weakly-periodic}]
	Direct application of Theorem \ref{thm:gen-poly=>not-auto-dens-1} with $f = g$ and $R = \NN_0$
\end{proof}

It is not a trivial matter to determine whether a given generalised polynomial is periodic away from a set of density $0$, although it can be accomplished by the techniques in \cite{BergelsonLeibman2007}, \cite{Leibman-2012}.\comment{Am I right in assuming that it is decidable to determine if a given semialgebraic set has non-empty interior?}\ In order to give explicit examples, we restrict to generalised polynomials of a specific form, which is somewhat more general than the one considered in Proposition \ref{prop:poly=>not-wp}. 


\begin{corollary}\label{thm:gen-poly-eqdist=>not-auto}
	Suppose that $q \colon \ZZ \to \RR$ is a generalised polynomial with the property that $\lambda g( an) \bmod{1}$ is equidistributed in $[0,1)$ for any $\lambda \in \QQ \setminus \{0\}$ and $a \in \NN$, and let $m \geq 2$. Then, the sequence $f(n) = \floor{ q(n) } \bmod{m}$ is not automatic.	
\end{corollary}
\begin{proof}
	Suppose $f(n)$ were automatic. By Theorem \ref{thm:main-weakly-periodic}, there exist $a \in \NN$ and $Z \subset \NN_0$ with $d^*(Z) = 0$, such that $f(an)$ is constant for $n \in \NN_0 \setminus Z$. Hence, there is some $0 \leq r < m$ such that $\frac{1}{m} q(an) \in \left[ \frac{r}{m}, \frac{r+1}{m} \right)$ for $n \in \NN_0 \setminus Z$, contradicting the equidistribution assumption.
\end{proof}


The uniform distribution of generalised polynomials has been extensively studied by \Haland\ \cite{Haland-1993},\cite{Haland-1994},\cite{HalandKnuth-1995}, and later very general theory was developed by Bergelson and Leibman \cite{BergelsonLeibman2007}, \cite{Leibman-2012}. In view of the the results in \cite{Haland-1993}, it is fair to say that for a ``generic'' generalised polynomial $q(n)$ is equidistributed modulo $1$. Hence, the assumptions on $q(n)$ in Corollary \ref{thm:gen-poly-eqdist=>not-auto} are relatively mild. 

To make the last remark precise, let us define the (multi-)set of coefficient of a generalised polynomial $q$ as follows. If $q(n) = \sum_{j} \alpha_j n^j$ is a polynomial, then the coefficients of $q(n)$ are the non-zero terms among the $\a_j$. If $q(n) = r_1(n) + r_2(n)$ or $q(n) = r_1(n) \cdot r_2(n)$, then the coefficients of $q(n)$ are the union of coefficients of $r_1(n)$ and $r_2(n)$. Finally, if $q(n) = \alpha \floor{ r(n)}$, then the coefficients of $q$ are the coefficients of $r(n)$ and $\alpha$. The set of coefficients will depend on the choice of representation of the generalised polynomial at hand; we fix one such choice. We cite a slightly simplified version of the main theorem of \cite{Haland-1993}.

\begin{theorem}\label{thm:equidistr-gen-poly}
	Suppose that $q(n)$ is a generalised polynomials, and all of the products of coefficients of $q(n)$ are $\QQ$-linearly independent. Then $q(n)$ is equidistributed modulo $1$.
\end{theorem}

As an example application, we may conclude that $\floor{ \sqrt{2} n \floor{ \sqrt{3} n} } \bmod{10}$ is not an automatic sequence.

\section{Sparse sets}\label{sec:Sparse}
\comment{Modulo the ``comments'', I am happy with what this section looks like now. Feel free to treat this as a final version, in particular suggest changes, edit, etc.}

In this section, we begin the investigation of sparse sequences. Here, we call a sequence $f \colon \NN_0 \to \{0,1\}$ \emph{sparse} if it is the characteristic function of a set with upper Banach density $ 0$\comment{This is diffetent than before; hope the change doesn't cause trouble.}. Note that for such sparse sequences, Theorem \ref{thm:main-weakly-periodic} conveys no useful information. Conversely, in light of Theorem \ref{thm:main-weakly-periodic}, to prove Conjecture \ref{conjecture:main}, it would suffice to verify it for sparse sequences --- this observation will be made precise in the Proof of Theorem \ref{thm:main-optimized} below.  \comment{This part can be shortened a bit.}

\comment{Possible stylystic choices for Conjectre (I know we have discussed this, but I'm having second thoughts).

1. Conjecture 

2. Conjecture A (independent numbering)

3. Conjecture B (numbering together with main Theorems)

4. Conjecture 1 

5. Main Conjecture

I think we settled on 2 (I may be misremembering), which is not bad, but has the disadvantage of using A to index two things.
}

To formulate our main result, it is convenient to introduce a piece of terminology.

\begin{definition}[Very sparse sets]\label{defverysparse}
	Let $k \geq 2$, $r \geq 0$. A basic very sparse set of rank $r$, base $k$, is a set of the form
	\begin{equation}\label{eq:v-sparse-def}
		E = \set{ [w_0 u_1^{l_1} w_1 u_2^{l_2} \dots u_r^{l_r} w_r]_k }{ l_1, \dots, l_r \in \NN_0},
	\end{equation}
	where $u_1,\dots,u_r \in \Sigma_k^*$, and $w_0,\dots,w_r \in \Sigma_k^*$. A very sparse set (of rank $r$, base $k$) is a finite union of basic very sparse sets (of rank $r$, base $k$). A sequence $f \colon \NN_0 \to \{0,1\}$ is very sparse if the set $\set{n \in \NN_0}{f(n) = 1}$ is very sparse. 
\end{definition}


\begin{lemma}
	Any very sparse sequence base $k$  is $k$-automatic.
\end{lemma}
\begin{proof}
	This is clear by explicit construction.
\end{proof}

\begin{remark}
	A subset of a sparse set is always a sparse set, but a subset of a very sparse set need not be a very sparse set.
\end{remark}

\begin{theorem}\label{thm:main-B2}
	Suppose that a sparse sequence $f \colon \NN_0 \to \{0,1\}$ is simultaneously $k$-automatic and generalised polynomial. Then, $f$ is very sparse. 
\end{theorem}
\begin{proof}[Proof of Theorem \ref{thm:main-optimized}, assuming \ref{thm:main-B2}]
	Suppose that $f \colon \NN_0 \to \RR$ is automatic and generalised polynomial. Let $p(n)$ be the periodic function such that the set $Z = \set{n \in \NN_0}{f(n) = p(n)}$ has $d^*(Z) = 0$, whose existence is guaranteed by Theorem \ref{thm:main-weakly-periodic}. Denote
	$$
	f'(n) = 
	\begin{cases}
		1, & \text{ if } f(n) = p(n), \\ 
		0, & \text{ otherwise.}
	\end{cases}
	$$ 
	Then (e.g.\ by Lemma \ref{lem:gen-poly-set-of-zeros}), $f'(n)$ is generalised polynomial and automatic. By Theorem \ref{thm:main-B2} $f'(n)$ is very sparse, so we have $\abs{ Z \cap [M,M+N) } = \sum_{n = M}^{M+N} f'(n) = O\bra{\log^r N}$ for some $r \in \NN_0$ as $N \to \infty$.
\end{proof}

We recall Theorem \ref{thm:main-dichotomy}, which is obtained as a corollary to the above structural description. More precisely, it will follow immediately from Proposition \ref{automverysparse} by repetition of the above argument.
\begin{theorem}\label{thm:main-C2}
	Fix $k \geq 2$. Either the set $\set{k^i}{i \in \NN_0}$ is generalised polynomial, or any $k$-automatic, generalised polynomial sequence is eventually periodic. 
\end{theorem}

While we believe that Conjecture \ref{conjecture:main} is true, hence there are no non-zero sparse sequences which are simultaneously generalised polynomial and automatic, we pause to give some examples of sparse generalised polynomials for which it is not a priori clear that they cannot be automatic. We postpone the proof of Theorem \ref{thm:main-B2} until later, and will delegate much of the technical work to Sections \ref{sec:S-GENPOLY} and \ref{sec:S-AUTO}.

\begin{lemma}\label{lem:gen-poly-set-of-zeros}
	If $h$ is an (unbounded) generalised polynomial, then 
	$$g(n) =
	\begin{cases}
	1, & \text{if } h(n) = 0, \\
	0, & \text{otherwise,}
	\end{cases}
	$$
	is a generalised polynomial. Likewise, for any $a < b$, 
	$$g'(n) =
	\begin{cases}
	1, & \text{if }  a \leq h(n) < b, \\
	0, & \text{otherwise,}
	\end{cases}
	$$
	is a generalised polynomial.	
\end{lemma}
\begin{proof}
 Because $h(n)$ takes countably many values for $n \in \NN_0$, there exists $\theta \in \RR$ such that $\theta h(n) \in \RR \setminus \QQ$ unless $h(n) = 0$. Now, a short computation verifies that 
 $$g(n) = 1 - \ceil{ \fp{\theta h(n)}}.$$
 For the additional part of the statement, after rescaling we may assume that $a = 0$ and $b = 1$. Then, $g'(n) = \ifbra{\floor{ h(n)} = 0}$, and we may repeat the construction above for $h'(n) = \floor{ h(n)}$.
\end{proof}
\begin{remark}
	In contrast with Lemma \ref{lem:gen-poly-set-of-zeros}, for a generalised polynomial $h$ and $a \in \RR$, the function $ \ifbra{ h(n) < a}$ will not generally be a generalised polynomial.
\end{remark}

\comment{Do we want the notion of a generalised polynomial set, automatic set, etc.? I thought we do, but now I no longer do. --jk We absolutely do! -- JB}

\begin{example}\label{ex:lin-rec-is-gp}
	Take $a \in \NN$ and let $(n_i)$ be the sequence given by $n_0 = 0,\ n_1 = 1$ and $n_{i+2} = a n_{i+1} + n_i$. Then $E = \set{n_i}{i \in \NN_0}$ is generalised polynomial. For $a = 1$, this is the set of Fibonacci numbers.
\end{example}
\begin{proof}
Let $\a \in \RR$ be the real number with continued fraction expansion $\a = [a;a,a,\dots]$, i.e.
$$
	\a = a+\cfrac{1}{a+\cfrac{1}{a+\cdots}} = \frac{a + \sqrt{a^2+4}}{2},
$$
and put $E' = \set{n \in \NN_0}{ \norm{n\a} < \frac{1}{2n} }$. Note that $E'$ is generalised polynomial by Lemma \ref{lem:gen-poly-set-of-zeros}, so it will suffice to show that the symmetric difference $E \triangle E'$ is finite.

By a classical theorem of Legendre (see e.g.\ \cite[Thm.\ 5.1]{Khintchine-book}), we have $E' \subset E$. Conversely, using a well-known formula for the error term in continued fraction approximations (see e.g.\ \cite[Thm.\ 3.1]{Khintchine-book}), we find
$$
	n_i \norm{n_i \a} = n_i^2 \sum_{l=0}^\infty \frac{(-1)^{l}}{n_{i+l} n_{i+l+1}} \xrightarrow[i \to \infty]{} \sum_{l=0}^\infty \frac{(-1)^l}{\a^{2l+1}} = \frac{1}{\a + 1/\a} = \frac{1}{\sqrt{a^2 + 4}},
$$
because $n_{i+1}/n_i \to \alpha$ as $i \to \infty$. Since $\frac{1}{\sqrt{a^2 + 4}} < \frac{1}{2}$, for sufficiently large $i$ we have $\norm{n_i \a} <  \frac{1}{2 n_i}$, whence $n_i \in E'$. 
\end{proof}

\begin{remark}\label{lawnmower}
	Similar argument works also for sequences given by recurrence $n_{i+2} = a_{i+2} n_{i+1} + n_i$, where $a_i \geq 2$ is a bounded sequence of integers. Then, $n_i$ is the sequence of denominators in the best reational approximations of a badly approximable real number.
\end{remark}

We are now ready to approach the proof of Theorem \ref{thm:main-B2}. The first ingredient is a structure theorem for sparse automatic sequences. 
It is covenient to introduce some additional terminology, namely a slight refinement of the notion of an \IPP-set introduced in Section \ref{sec:DEF}, allowing for a variable sequence of shifts.
\begin{definition}[\IPS-set]\label{def:IPS-set}
	For a sequence $(n_i)_{i \in \NN} \subset \NN$ and shifts $(N_t)_{t \geq 1} \subset \NN_0$, the corresponding set of \emph{shifted finite sums} is 
\begin{equation}\label{eq:IPS-def}
\FS(n_i;N_t) = \set{ n_\a + N_t }{ t \in \NN,\ \a \subset [t]},
\end{equation}  
	where $n_\a = \sum_{i \in \a}n_i$.	Any set containing $\FS(n_i;N_t)$ for some $(n_i),\ (N_t)$ is called an \emph{\IPS-set}.
\end{definition}
\comment{Is this the correct way of writing sequences?}

\begin{remark}
	If $N_t = 0$ (resp.\ $N_t = \mathrm{const.}$) then the set $\FS(n_i;N_t)$ in \eqref{eq:IPS-def} is an \IP-set (resp.\ \IPP-set). In general, an \IPS-set need not be an \IPP-set.
\end{remark}

\begin{remark}
	An equivalent (though apparently weaker) definition of an \IPS-set can be phrased as follows: A set $E \subset \NN$ is \IPS-if there exist sequences $(n_i)_{i \in \NN} \subset \NN$, $(N_t)_{t \geq 1} \subset \NN_0$ and $(t_0(\a))_{\a \in \cF} \subset \NN$ such that $n_\a + N_t \in E$ for any $\a \in \cF$ and $t \geq t_0(\a)$. 
\end{remark}

We are now ready to phrase the structure theorem alluded to before, which is interesting in its own right. The proof is given in Section \ref{sec:S-AUTO}.

\begin{theorem}\label{thm:Structure-Auto}
	Let $f \colon \NN_0 \to \{0,1\}$ be a sparse automatic sequence. Then, either $f$ is very sparse, or there exists an \IPS-set $E$ such that $f(n) = 1$ for all $n \in E$. 
	
\end{theorem}

In light of Theorem \ref{thm:Structure-Auto}, combined with Theorem \ref{thm:BergelsonLeibman}, the relevance of the behaviour of nilsequences along \IPS-sets becomes clear. The following theorem is the main result of Section \ref{sec:S-GENPOLY}.

\begin{theorem}\label{thm:GP-vs-IPS}
	Let $(X,T)$ a minimal nilsystem, and let $x \in X$. Suppose that $S \subset X$ is semialgebraic, and that the set $\set{ n \in \NN_0}{ T^n x \in S}$ is \IPS. Then $\inter S \neq \emptyset$.
\end{theorem}

\begin{proof}[Proof of Theorem \ref{thm:main-B2} assuming \ref{thm:Structure-Auto} and \ref{thm:GP-vs-IPS}]

Let $f \colon \NN_0 \to \{0,1\}$ be sparse, automatic and generalised polynomial. By Theorem \ref{thm:BergelsonLeibman}, there exists a minimal nilsystem $(X,T)$ and a semialgebraic set $S \subset X$, such that $f(n) = \ifbra{ T^n x \in S}$.

Suppose for the sake of contradiction that $f$ is not very sparse. Then, by Theorem \ref{thm:Structure-Auto}, the set $\set{ n \in \NN_0}{ T^n x \in S}$ is \IPS. Hence, by Theorem \ref{thm:GP-vs-IPS}, $\inter S \neq \emptyset$, and in particular $\mu_X(S) > 0$. On the other hand, by the Ergodic Theorem \ref{thm:ergo-thm-uniform}, we have
$ \mu_X(S) = d\bra{\set{ n \in \NN_0}{f(n) = 1}} = 0,$ which is a contradiction.
\end{proof}

\comment{Don't forget to say that nilsystems come with natural measures, $\mu_X$.}

\comment{We should mention at some point that if $S$ is semialgebraic, then $\mu(\partial S) = 0$.}
\section{Sparse generalised polynomials}\label{sec:S-GENPOLY}

\newcommand{\PHI}{P}
\newcommand{\PSI}{Q}

In this section, we prove Theorem \ref{thm:GP-vs-IPS}. In the process, we obtain results concerning the distribution of fractional parts of polynomial sequences, such as Proposition \ref{lem:variable-separation} and \ref{lem:fp-closure}, which are possibly of independent interest.

It will be slightly more convenient to reformulate Theorem \ref{thm:GP-vs-IPS} as follows.

\begin{theorem}\label{thm:S:main-1}
	Let $G/\Gamma$ be a nilmanifold, $g \in G$ whose action of $G/\Gamma$ is minimal, $a \in G$, and let $S \subset G/\Gamma$ be a semialgebraic set with $\inter S = \emptyset$.  
	Then, there does not exist an \IPS-set $E$ such that 
	\begin{equation}
	\label{eq:S:main-1}
	g^{n } a \Gamma \in S, \qquad n \in E.
	\end{equation}
\end{theorem}

\subsection*{Preliminaries}
In dealing with \IPS- sets, the following Lemma will be useful.

%
%

\newcommand{\fpp}[2]{\fp{#1}_{#2}}
\newcommand{\ipp}[2]{\left[ #1 \right]_{#2}}

The following lemma is standard. For proof, see e.g. \cite[Lemma 2.5]{Leibman-2005b}, \cite[Lemma 3.4]{Leibman2005}, or \cite[Lemma 1.11]{Zorin-Kranich2013}.
\comment{ We do *not* include the succinct proof for the sake of pleasing Ben.}
%
\begin{lemma}\label{lem:G2-not-complemented}
	Let $G$ be a nilpotent group (not necessarily Lie), and let $H < G$ be a subgroup of $G$ such that $H G_2 = G$. Then $H = G$.
\end{lemma}
%
%

The following fact is standard. We include a succinct proof for the convenience of the reader because we were not able to locate it in the literature in this form. Alternatively, we could also use a different factorisation which appears e.g.\ in \cite[Corollary 4.26]{Zorin-Kranich2013}. \comment{Would be really nice to include a reference...} 
We will need the fact that if $G/\Gamma$ is a connected nilmanifold, then $G_i = (G^{\circ} \cap G_i)(\Gamma \cap G)$ for all $i$, which can be shown by a simple induction.
\comment{Reference would be nice...}

\begin{lemma}\label{lem:poly-seq-in-G-o}
	Let $G/\Gamma$ be a connected nilmanifold, and let $g\colon \ZZ \to G$ be a polynomial sequence. Then, there exists a polynomial sequence $h \colon \ZZ \to G$ with coefficients in $G^{\circ}$, such that $g(n)\Gamma = h(n)\Gamma$.
\end{lemma}
\begin{proof}
	Let $s$ be the step of $G$. We may write $g(n)$ as
	\begin{equation}
	\label{eq:435}
		g(n) = a_0^{\binom{n}{0}} a_1^{\binom{n}{1}} \dots a_s^{\binom{n}{s}},
	\end{equation}

	where $a_i \in G_i$. We show by induction on $r$ that there exists a polynomial sequences $h_r,\gamma_r \colon \ZZ \to G$ taking values in $G^\circ$ and $\Gamma$ respectively such that $$h_r(n)^{-1}g(n) \gamma_r^{-1}(n) = e ,$$ for $ n = 0,1,\dots,r.$

	Suppose the claim holds for $r-1$. Replacing $g(n)$ with $h_{r-1}(n)^{-1}g(n) \gamma_{r-1}^{-1}(n),$ we may assume $g(0) = \dots = g(r-1) = e$. Hence, in \eqref{eq:435} we have $a_0 = \dots = a_{r-1}$. We may write $a_{r} = b_r \gamma_r$, where $b_r \in G_r \cap G^0$ and $\gamma_r \in G_r \cap \Gamma$. Now $b_{r}^{-\binom{n}{i}} g(n) \gamma_{r}^{-\binom{n}{r}} = e$ for $n = 0,1,\dots,r$.
\end{proof}

We will need also the following basic fact about invariant algebraic sets. 
\begin{lemma}\label{lem:invariant=>strongly-invariant}
	Suppose that $A \subset \RR^d$ is an algebraic set, and $\PHI \colon \RR^d \to \RR^d$ is a polynomial map. If $\PHI^{-1}(A) \subset A$, then also $\PHI(A) = A$.
\end{lemma}
\begin{proof}
	Consider the descending sequence of nonempty algebraic sets
	\begin{align}
	\label{eq:865-1}
		A \supseteq \PHI^{-1}(A) \supseteq \PHI^{-2}(A) \supseteq \PHI^{-3}(A) \supseteq \dots
	\end{align}	 
	
 	By Hilbert's Basis Theorem, there is $n$ such that $\PHI^{-n}(A) = \PHI^{-n-1}(A)$, whence $\PHI(A) = A$.
\end{proof}

\subsection*{Initial reductions}

Our proof of Theorem \ref{thm:S:main-1} is indirect: we exploit the existence of orbit satisfying \eqref{eq:S:main-1} to construct algebraically invariant objects whose existence is inreasingly difficult to sustain, up to the point when we obtain a blatant contradiction. We begin with a relatively straightforward reduction.

\begin{step}\label{prop:S:main-2}
	Assume Theorem \ref{thm:S:main-1} is false. Then, there exists a connected nilmanifold $G/\Gamma$, $g \in G$ whose action on $G/\Gamma$ is minimal, as well as a piecewise algebraic subset $S \subset G/\Gamma$ with $\inter S = \emptyset$ and an \IP-set $E \subset \NN$ such that
	\begin{equation}
	\label{eq:S:main-2}
	g^{n } \Gamma \in S, \qquad n \in E.
	\end{equation}
\end{step}
\begin{proof}
	Suppose that Theorem \ref{thm:S:main-1} fails, so that \eqref{eq:S:main-1} holds. 	
	We first address the issue of connectedness. There is an integer $d$, such each connected component of $G/\Gamma$ contains a point $h\Gamma$ with $h^d \in \Gamma$ (\cite[2.7]{Leibman2005}). The group $\tilde \Gamma$ generated by $\Gamma$ and one $h \in G$ as above has $[\tilde \Gamma:\Gamma] < \infty$ (\cite[1.14]{Zorin-Kranich2013}). Replacing $\Gamma$ with $\tilde \Gamma$ and $S$ with its image under projection $G/\Gamma \to G/\tilde \Gamma$, we may assume that $G/\Gamma$ is connected.	

	
	\comment{I'm pretty sure this works, but please see if you are convinved. Lemma 1.14 in ZK-thesis is of some relevance here. Suggestion: leave it like this, and change iff sb complains. --JK}
	
	 Let $(n_i)_{i \in \NN}$, $(N_t)_{t \geq 1}$ and $(t_0(\a))_{\a \in \cF}$ be such that $n_\a + N_t \in E$ for $t \geq t_0(\a)$, so that
	$$
		g^{n_{\a} + N_t} a \Gamma \in S. 
	$$
	Replacing $S$ with its closure, which is again a semialgebraic set with empty interior, we may assume that $S$ is closed. Restricting $N_t$ to a subsequence if necessary, we may assume that $g^{N_t} a \Gamma \to b \Gamma$ for some $b \in G$ as $t \to \infty$. 
	It follows that for any \IP-set $E' \subset \FS(n_i)$ we have 
	$$
		g^{n } b \Gamma = \lim_{t \to \infty} g^{n + N_t} a \Gamma  \in S, \qquad n \in E'. 
	$$
	Finally, put $\Gamma' = b \Gamma b^{-1}$ and let $S'$ be the projection to $G/\Gamma'$ a copy of $S$ in $G$. We find that $g^n \Gamma' \in S'$ for $n \in E'$, as needed.
\end{proof}
	  
\subsection*{Fractional parts and limits}
	  
Let $G/\Gamma$ be a nilmanifold with specified \Malcev\ basis. A technical difficulty in dealing with fractional parts in $G$ stems from the fact that, for a sequence $g_n \in G$, the limit of $g_n \Gamma$ does not determine the limit of $\fp{g_n}$ if the limit lies on the boundary of the fundamental domain. The following lemma partially overcomes this difficulty. 
		  
\begin{lemma}\label{lem:conv-of-fp-to-0}
	Let $G/\Gamma$ be a connected nilmanifold. Then, for any $(g_n)_{n \geq 0} \subset G$ with $\lim_{n \to \infty} g_n \Gamma = e \Gamma$, there exists a choice of \Malcev\ basis $\cX$ and a subsequence $(x_{n})_{n \in L}$, such that $\lim_{n \to \infty, n \in L} \fp{x_n} = e$.
	
	Likewise, for any $(g_\a)_{\a \in \cF}$, there exists a choice of \Malcev\ basis $\cX$ and \IP-ring $\cG$ such that $\iplim_{\a \in \cG} \fp{g_\a} = e$.
\end{lemma}
\begin{proof}
	By induction, we prove that for each $r$, there exists a \Malcev\ basis $\cX = (X_1,X_2,\dots,X_d)$ such that (possibly after passing to a subsequence) we have 
	$$\lim_{n \to \infty} \fp{x_n} = h$$
	such that $\tau_i(h) = 0$ for $i < r$. For $r = 1$, the claim is trivially satisfied, and the claim for $r = s+1$ implies the statement of the lemma, where $s$ is the nilpotency class of $G$.
	
	If the claim holds for $r$, then we may write
	$$
		g_n = \exp(t_r(n) X_r) \exp(t_{r+1}(n) X_{r+1}) \dots \exp(t_{d}(n) X_{d})  \gamma(n),
	$$
	where $t_i(n) \in [0,1)$, $\gamma(n) \in \Gamma$, and $t_i(n) \to \tau_i(h)$ as $i \to \infty$. Since $h \in \Gamma$, we have $\tau_r(h) \in \{0,1\}$. If $\tau_r(h) = 0$, we are done. Otherwise, we may write
	\begin{align*}
	g_n 
	&= \exp( (1-t_r(n))(- X_r) ) \exp(X_r) \exp(t_{r+1}(n) X_{r+1}) \dots \exp(t_{d}(n) X_{d})  \gamma(n) \\
	&= \exp( (1-t_r(n))(- X_r) ) \exp(t_{r+1}'(n) X_{r+1}) \dots \exp(t_{d}'(n) X_{d})  \gamma'(n),
	\end{align*}
	where $t_i'(n) \in [0,1)$ and $\gamma'(n) \in \Gamma$. Replacing $\cX$ with $\cX' =(X_1',\dots,X_d')$ with $X_r' = -X_r$ and $X_i' = X_i$ for $i \neq r$, we find the sought \Malcev\ basis for $r+1$.
	
	The proof of the additional part for \IP-limits is fully analogous.
\end{proof}

We are now ready to perform the first substantial step in the proof of Theorem \ref{thm:S:main-1}, where we construct an algebraic set inside $G$ with surprising invariance properties. This is helpful largely because working in $G$ is easier than in $G/\Gamma$.

\begin{step}\label{prop:S:main-3}
	Suppose that Theorem \ref{thm:S:main-1} is false. Then there exists a connected nilmanifold $G/\Gamma$ with \Malcev\ basis $\cX$, $g \in G$ whose action on $G/\Gamma$ is minimal, an algebraic set $R \subset G^{\circ}$ with $\inter R = \emptyset$, and an \IP-set $E = \FS(n_i)$ such that $R$ is preserved under the operations
\begin{equation}
		\PHI_\a \colon G \to G, \qquad x \mapsto g^{n_\a} x \ip{ g^{n_\a}}^{-1}
\end{equation}  
	for all $\a \in \cF$. Additionally, $\PHI_\b \circ \PHI_\a = \PHI_{\a \cup \b}$ for all $\a < \b$.
\end{step}

\begin{proof}
	Let notation be as in the conclusion of Step \ref{prop:S:main-2}, but we reserve the right to replace $E = \FS(n_i)$ with an \IP-subset. 
	
	By the \IP-recurrence theorem (\cite[Lemma 9.10]{Furstenberg1981}), after replacing $E$ with a smaller \IP-set, we may assume that $\iplim_{\a} g^{n_\a} \Gamma = e \Gamma$. Hence, by Lemma \ref{lem:conv-of-fp-to-0} there exists a choice of \Malcev\ basis such that, possibly after shrinking $E$ again, we have 
\begin{equation}
	\label{eq:527-1}
	\iplim_{\a \in \cF} \fp{ g^{n_\a} } = e.
\end{equation}  
	
For fixed $\b$ and sufficiently large $\a > \b$, we aim to compute $\fp{g^{n_{\a \cup \b}}}$ in terms of $g^{n_\b}$ and $\fp{g^{n_\a}}$. We begin by noting that
$$
	\fp{g^{n_{\a \cup \b}}} 
	= \fp{ g^{n_\b} \fp{g^{n_\a}}} 
	= g^{n_\b} \fp{g^{n_\a}} \ip{g^{n_\b} \fp{g^{n_\a}} }^{-1}
 	= g^{n_\b} \fp{g^{n_\a}} \ip{g^{n_\b}}^{-1} \gamma(\a,\b),
$$	
where $\gamma(\a,\b)$ is bounded by a constant dependent at most on $\b$. Using a diagonalising argument to pass to an \IP-subset of $E$, we may assume that $\gamma(\a,\b) = \gamma(\b)$ does not depend on $\a$. Letting $\a \to \infty$ we now have
$$
	\iplim_{\a \in \cF} \fp{g^{n_{\a \cup \b}}} 
= 	\fp{ g^{n_\b} } \gamma(\b).
$$ 
	If $\b$ is sufficiently large, this implies that $\gamma(\b) = e$. Passing to an \IP-subset, we may assume that $\gamma(\beta) = e$ for all $\b \in \cF$. Hence, we have:
\begin{align}
	\label{eq:527-3}
	\fp{g^{n_{\a \cup \b}}} =  g^{n_\b} \fp{g^{n_\a}} \ip{g^{n_\b}}^{-1} = \PHI_\b(\fp{g^{n_\a}}),
\end{align}
	for any $\a,\b \in \cF$ with $\b < \a$.	
%
	Because $\PHI_\b \circ \PHI_\a$ and $\PHI_{\b \cup \a}$ agree on at least one point (such as $\fp{g^{n_{\lambda}}}$ for $\lambda > \alpha$), a short manipulation shows that $\PHI_\b \circ \PHI_\a = \PHI_{\b \cup \a}$ everywhere. We will use this fact several times.

	Denote by $\widetilde{S} \subset G^{\circ}$ the Zariski closure of the copy of $S$ contained in the fundamental domain for $\cX$. Then, $\tilde S$ is a algebraic subset of $G^{\circ}$ with $\inter \widetilde{S} = \emptyset$. 

	 For any $\a$, let $R_\a$ denote the set of $x \in \tilde{S}$ such that $\PHI_{\b}(x) \in \tilde S$. Clearly, $R_\a$ are algebraic, and by \eqref{eq:527-3} we have $\fp{ g^{n_\a} } \in R_\b$ if $\a > \b$. Finally, let $R = \bigcap_{\a \in \cF} R_\a$.
	
	Because of Hilbert's Basis Theorem, there is finite collection $\cB \subset \cF$ such that $R = \bigcap_{\b \in \cB} R_\b$. Note that $R$ is non-empty as $\fp{ g^{n_\a} } \in R$ if $\a > \b$ for all $\b \in \cB$. It remains to show that $R$ is preserved under $\PHI_\a$ for $\a$ large enough.
	
	Suppose that $x \in R$ and $\a > \b$ for all $\b \in \cB$; we will show that $\PHI_\a(x) \in R$. For this, it suffices to check that $\PHI_\a(x) \in R_\b$ for $\b \in \cB$, that is $\PHI_\a(x) \in \tilde{S}$ and $\PHI_\b( \PHI_\a(x)) \in \tilde S$. The first condition follows immediately from $x \in R_\a$, and the latter from $x \in R_{\a \cup \b}$.
\end{proof}

\subsection*{Fractional parts and polynomials}

Our next aim is to show that the set $R$ constructed in Step \ref{prop:S:main-3} needs to be invariant under a wider range of operations. For this purpose, we study the possible algebraic relations between polynomials and their fractional parts.

\begin{lemma} \label{lem:variable-separation}
	Let $G/\Gamma$ be a nilmanifold equipped with a \Malcev\ basis $\cX$. Let $g,h\colon \ZZ \to G^{\circ}$ be polynomial sequences. 
	Suppose that $E$ is an \IP-set and $\PHI \colon G^{\circ} \times G^{\circ} \to \RR$ is a polynomial such that 
	\begin{equation}
	\label{eq:782-1}
	\PHI\bra{ g(n), \fp{h(n)} } = 0, \qquad n \in E.
	\end{equation}
	Then, there exists an \IP-set $E' \subset E$ such that
	\begin{equation}
	\label{eq:782-9}
	\PHI\bra{g(m),\fp{h(n)} } = 0, \qquad m \in \ZZ,\ n \in E'.
	\end{equation}
\end{lemma}
\begin{remark}
	If above $g$ is instead a polynomial sequence $\RR \to G^o$, then the same argument gives $\PHI\bra{g(t),\fp{h(n)} } = 0$ for $t \in \RR$, $n \in E'$.
\end{remark}
\begin{proof}
	By Lemma \ref{lem:poly-seq-in-G-o}, we may assume that $g,h$ have coefficients in $G^o$. Note that there exists a polynomial $\PSI \colon \RR \times G \to \RR$ such that $\PHI(g(m),y) = \PSI(m,y)$. Expand $\PSI$ as $ \PSI(x,y) = \sum_{k} x^k \PSI_k(y)$, for some polynomials $\PSI_k \colon G \to \RR$. 

	Looking at the leading term in equation 
	\begin{equation}
	\label{eq:782-2}
		\sum_{k = 0}^d n^k \PSI_k\bra{ \fp{h(n)}  } = 0,
	\end{equation}
	we conclude that 
	\begin{equation}
	\label{eq:782-25}
	\lim_{E \ni n \to \infty} \PSI_d\bra{ \fp{h(n)}  } = 0.
	\end{equation}
	Take $n_i$ such that $n_\a \in E$ for $\a \in \cF$.	In particular, as a special case of \eqref{eq:782-25} for any fixed $\b \in \cF$ we have
	\begin{equation}
	\label{eq:782-3}
	\iplim_{\a \in \cF} \PSI_d \bra{ \fp{h\bra{n_\b + n_\a}}}  = 0.		
	\end{equation}	
	By the \IP\ polynomial recurrence theorem for nilrotation (see \cite[Theorem D]{Leibman-2005b}) 
	there exists an \IP-ring $\cF_\b$ (which may be chosen to be contained in any previously specified IP-ring) such that 
	\begin{equation}
	\label{eq:782-4}\iplim_{\a \in \cF_\b} h(n_\a + n_\b) \Gamma = h(n_\b) \Gamma.
	\end{equation}

	A slight technical difficulty stems from the fact that $g \mapsto \fp{g}$ is discontinuous. To overcome this problem, let $\e > 0$ be sufficiently small that if $x \in G$ lies in the fundamental domain, then $d(x, x \gamma) \geq \e$ for all $\gamma \in \Gamma$. After replacing $E$ with an \IP-subset $E'$, we may assume that all the points $\fp{h(n)}$ for $n \in E'$ lie within a ball of radius $\e/10$, which will be useful shortly.

	Applying \eqref{eq:782-4}, possibly after refining $\cF_\b$ further, we find
	\begin{equation}
	\label{eq:782-5}\iplim_{\a \in \cF_\b} \fp{ h(n_\a + n_\b)} = 
	\fp{ h(n_\b) }\gamma(\b),	
	\end{equation}
	where $\gamma(\b) \in \Gamma$. Now, because two limit points of $\fp{h(n_\a)}$ cannot differ by a factor of $\gamma \in \Gamma \setminus \{e\}$, we conclude that 
		\begin{equation}
	\label{eq:782-6}\iplim_{\a \in \cF_\b} \fp{ h(n_\a + n_\b)} = \fp{ h(n_\b)}.
	\end{equation}
	
	 Because $\PSI_d$ is continuous, \eqref{eq:782-6} yields $\PSI_d \bra{ \fp{ h(n_\b)}} = 0$, where we recall that $\b$ was arbitrary. Reasoning inductively, we obtain similarly $\PSI_k \bra{ { h(n_\b)} } = 0$ for all $k$ and $\b \in \cF$. Consequently, $\PSI\bra{m, \fp{ h(n_\b)}} = 0$ for all $m \in \ZZ$ and $n \in E'$, as desired.
\end{proof}

\begin{step}\label{prop:S:main-4}
	Suppose that Theorem \ref{thm:S:main-1} is false. Then there exists a connected nilmanifold $G/\Gamma$ with \Malcev\ basis $\cX$, $g \in G$ whose action on $G/\Gamma$ is minimal, an algebraic subset $R \subset G^{\circ}$ with $\inter R = \emptyset$ and an \IP-set $E$ such that $R$ is preserved under the operations
\begin{align*}
	x \mapsto g^{m} x g^{-m}, \qquad m \in \ZZ \\
	x \mapsto x \fp{g^{n}}, \qquad n \in E.
\end{align*}  
\end{step}
\begin{proof}
	Assume notation as in the conclusion of Step \ref{prop:S:main-3}. 
	
	Let us consider for $x \in R$ the set
	$$
		M_x = \set{ (m, n) \in \ZZ \times \ZZ }{ g^m x g^{-m} \fp{ g^n } \in R}.
	$$
	Because multiplication in $G^{\circ}$ is given by polynomials in coordinates and $R$ is algebraic, the condition $(m,n) \in M_x$ is equivalent to a system of polynomial equations in $g^m x g^{-m}$ and $\fp{ g^n }$. Because $(n,n) \in M_x$ for all $n \in R$, Lemma \ref{lem:variable-separation} implies that $M_x$ contains $\ZZ \times E'$ for some \IP-set $E'$ (which can be chosen inside any \IP-subset of $E$). Iterating this argument, for any finite set $S \subset R$, we find an \IP-set $E'$ such that $\ZZ \times E' \subset \bigcap_{x \in S} M_x$.
	
	Consider also for $x \in R$ the set 
	$$
		P_{x} = \set{ h \in G^{\circ} }{ g^m x g^{-m} h \in R \text{ for all } m \in \ZZ },
	$$	
	so that $n \in M_x$ precisely when $\fp{g^n} \in M_x$. Because $P_x$ are algebraic, by the Hilbert's Basis Theorem there exists a finite set $R_0 \subset R$ such that 
	$$
		P := \bigcap_{ x \in R } P_{x} = \bigcap_{ x \in R_0 } P_{x}.
	$$
Now, it follows from the previous considerations that there exists an \IP-set $E' \subset E$ such that $\ZZ \times E' \subset M_x$ for $x \in R_0$, whence $\fp{ g^{n} } \in P$. In particular, $g^m R g^{-m} \fp{g^n} = R \fp{g^n} = R$ for $n \in E'$ and $m \in \ZZ$, so $R$ is preserved by $x \mapsto g^{m} x g^{-m}$ and $x \mapsto x \fp{g^{n}}$.
\end{proof}

\subsection*{Group generated by fractional parts}

We are now ready to take the last step in the proof of Theorem \ref{thm:S:main-1}. Our strategy is to show that the set $R$ appearing in Step \ref{prop:S:main-4} in invariant under multiplication by all of $G^{\circ}$, which is clearly absurd. In order to facilitate this strategy, we need the following fact.

\begin{proposition}\label{lem:fp-closure}
	Let $G/\Gamma$ be a nilmanifold with a fixed choice of \Malcev\ coordinates, and let $g(n)$ be a polynomial sequence equidistributed in $G$. Suppose that $H < G$ is a Lie subgroup, such that $\fp{g(n)} \in H$ for infinitely many $n$. Then, $H \supset G^{\circ}$. 
\end{proposition}

\begin{step}\label{prop:S:main-5}
	Theorem \ref{thm:S:main-1} holds true.
\end{step}
\begin{proof}[Proof (assuming Propositon \ref{lem:fp-closure})]
	Suppose Theorem \ref{thm:S:main-1} were false, and assume notation as in conclusion of Step \ref{prop:S:main-4}. Denote $$H = \set{ h \in G^{\circ}}{ R h \subset R} = \bigcap_{x \in R} \set{ h \in G^{\circ} }{ x h \in R}.$$
	
	It is immediate from the definition that $H$ is algebraic and closed under multiplication and (by Lemma \ref{lem:invariant=>strongly-invariant}) taking inverses. Hence, $H$ is a Lie subgroup of $G^{\circ}$.
Applying Proposition \ref{lem:fp-closure} to $\fp{g^{n}} \in H$ ($n \in E$), we conclude that $H = G^{\circ}$. In particular, $R = G^{\circ}$, which is in contradiction with $\inter R = \emptyset$.
%
%
\end{proof}
%

\begin{proof}[Proof of Proposition \ref{lem:fp-closure}]
	Let $E$ be an infinite set with $\fp{g(n)} \in H$ for $n \in E$.
	
	We begin with the case when $G$ is connected. Using Lemma \ref{lem:G2-not-complemented}, it will suffice to show that $HG_2 = G$.	Because of connectivity, $G/G_2$ can be identified with $\RR^d$ in such a way that $\Gamma$ maps to $\ZZ^d$. Let $\pi \colon G \to \RR^d$ be the corresponding projection. Put $\theta(n) = \pi( g(n) )$, $\xi(n) = \pi( \fp{g(n)} ) = \fp{ \theta(n) }$ and $V = \pi(H) < \RR^d$. Recall that $\xi(n)$ is bounded, $\theta(n)$ is polynomial, and $\xi(n) \in V$. We need to show that $V = \RR^d$.
	
%

	Suppose for the sake of contradiction that $V \neq \RR^d$. Then, there exists a non-trivial linear relation:
	\begin{equation}
	\label{eq:324-1}	
		\sum_{i=1}^d \lambda_i \fp{\theta_i(n)}  = 0, \qquad n \in E,
	\end{equation}
	where $\lambda_i \in \RR$, not all $ = 0$. The above may be rewritten as
	\begin{equation}
	\label{eq:324-2}	
		\sum_{i=1}^d \lambda_i \ip{\theta_i(n)} = \sum_{i=1}^d \lambda_i \theta_i(n), \qquad n \in E.
	\end{equation}
	Writing 
	$$ - \sum_{i=1}^d \lambda_i \theta_i(n) = \sum_{j=0}^D \kappa_j n^j,$$
	we thus obtain
	\begin{equation}
	\label{eq:324-3}
		\sum_{i=1}^d \lambda_i \ip{\theta_i(n)} + \sum_{j=0}^D \kappa_j n^j = 0 , \qquad n \in E.
	\end{equation}
	which amounts to saying that there is a non-trivial linear relation between the vectors $(\ip{\theta_i(n)})_{n \in E}$ for $i = 1,\dots,d$ and $(n^j)_{n \in E}$ for $j = 0,1,\dots,D$. Because these vectors are integer-valued, if one such relation exists, then there is also a relation with integer coefficients. Take one such relation
	\begin{equation}
	\label{eq:324-7}
		\sum_{i=1}^d l_i \ip{\theta_i(n)} + \sum_{j=0}^D k_j n^j = 0 , \qquad n \in E,
	\end{equation}
	where $l_i \in \ZZ$ and $k_j \in \ZZ$ for all $i,j$, and not all of $l_i,k_j$ are $0$. 	%
%
	Dropping the integer parts in \eqref{eq:324-7} at the cost of introducing a bounded error, we recover
	\begin{equation}
	\label{eq:324-4}
		\sum_{i=1}^d l_i {\theta_i(n)} + \sum_{j=0}^D k_j n^j = O(1), \qquad n \in E.
	\end{equation}
	This is only possible if the left hand side of \eqref{eq:324-4} is constant (as a polynomial, hence for $n \in \ZZ$).	It follows that $\theta(n) \bmod{\ZZ^d}$ for $n \in \ZZ$ takes values in the proper sub-torus $\set{x \in \RR^d/\ZZ^d}{ \sum l_i x_i = c }$, which contradicts equidistribution of $g(n)\Gamma$.
	
	Derivation of the general case is now a standard reduction. Passing to a subsequence, we may assume that $g(n) \Gamma$ all lie in the same connected compontent for $n \in E$.  \comment{I think we spent quite a while showing that this can be accomplished by intersecting $E$ with an arithmetic progression. This is true, but completely irrelevant.} Replacing $\Gamma$ with a conjugate $b \Gamma b^{-1}$ (and correspondingly replacing $g(n)$ with $g(n)b^{-1}$), we may assue that this is the component containing $e \Gamma$. By Lemma \ref{lem:poly-seq-in-G-o}, we may assume that the coefficients of $g(n)$ lie in $G^{\circ}$. It remains to apply the previous reasoning. 
\comment{This is a small cheat. It is OK to say that $g(n)$ is now polynomial in $G^{\circ}$, but it is the one place in the entire paper when this is NOT a polynomial with respect to the lower central series. The catch is that $(G^{\circ})_{i} \not \equiv G^\circ \cap G_i$. However, there is absolutely no problem in applying the same ARGUMENT, the only difference is that the quotient $G/G_2$ should now be $G^{\circ}/ (G_2 \cap G^\circ)$ rather than $G^{\circ}/ G^\circ_2$.} 
\end{proof}

\section{Sparse automatic sets}\label{sec:S-AUTO}

The aim of this chapter is to provide proofs of the several combinatorial results on automatic sequences that have been used in the previous sections, specifically Theorem \ref{thm:Structure-Auto}. We begin with some preparatory results. 
For a $k$-automaton $\mathcal{A}=(S,s_{\bullet},\delta,\{0,1\},\tau)$ with output $\{0,1\}$, we call a state $s\in S$ \emph{promising} if there exists a word $w\in \Sigma_k^*$ such that $\tau(\tilde{\delta}(s,w))=1$.

\begin{proposition}\label{automclassthm} Let $\mathcal{A}=(S,s_{\bullet},\delta,\{0,1\},\tau)$ be a $k$-automaton with output $\{0,1\}$. Then either 
\begin{enumerate}
\item [(i)] there exists a promising state $s\in S$ and $v_1,v_2\in \Sigma_k^*$ such that $v_1\neq v_2$, $|v_1|=|v_2|$, and $\tilde{\delta}(s,v_1)=\tilde{\delta}(s,v_2)=s$
\end{enumerate} or \begin{enumerate}
\item[(ii)]  the set of $w\in \Sigma_k^*$ such that $\tau(\tilde{\delta}(s_{\bullet},w))=1$ is a finite union of sets of the form $$\{v_1w_1^{l_1}\cdots v_tw_t^{l_t} \mid l_1,\ldots,l_t\geq 0 \}$$ for $t\geq 1, w_i,v_i\in \Sigma_k^* \text{ for } 1\leq i \leq t,$ and $v_i\neq \epsilon \text{ for } 2\leq i \leq t.$
\end{enumerate}
\end{proposition}

\begin{proof} Let $S^+ \subset S$ denote the set of promising states.
We construct a directed multigraph $G$ with vertex set $S^+$, and for each $s \in S^+,\ v \in \Sigma_k$ and edge from $s$ to $\delta(s,v)$. For $s_1,s_2 \in S^+$, write $s_1 \sim s_2$ if $s_1$ and $s_2$ are in the same strongly connected component of $G$, i.e.\ there exist $v_1, v_2\in \Sigma_k^*$ such that $\tilde{\delta}(s_1,v_1)=s_2$, and $\tilde{\delta}(s_2,v_2)=s_1$. We see that $\sim$ is an equivalence relation. For $s_1, s_2$ in $S^+$ we further write that $s_1 \prec s_2$ if there exist $v_1, v_2 \in \Sigma_k^*$ such that $v_1\neq v_2$, $|v_1|=|v_2|$ and $\tilde{\delta}(s_1,v_1)=\tilde{\delta}(s_1,v_2)=s_2$. The relation $\prec$ is a partial weak order and is invariant with respect to $\sim$ (i.e.\ if $s_1\sim s'_1$, $s_2\sim s'_2$, and $s_1\prec s_2$, then $s'_1\prec s'_2$), and therefore induces a partial weak order on the set of connected components $S^+/\mathord\sim$, which we continue to denote by $\prec$. Assume now that the condition (i) does not hold. This means precisely that $\prec$ is a partial strict order (on $S^+$ or $S^+/\mathord\sim$). 

%

Recall that a \emph{cycle graph} of length $n \geq 1$ is a graph $C$ on $n$ vertices $\{v_1,\dots,v_n\}$ with exactly $n$ edges, going from $v_i$ to $v_{i+1}$ for $1 \leq i \leq n$, where $v_{n+1} := v_1$.
\begin{claim*}
	Any strongly connected component $C$ of $G$ is either a cycle or consists of a single vertex.
\end{claim*}

\begin{proof}[Proof of the claim] 
If $C$ has a single vertex, then there is nothing to prove, so assume this is not the case. Strong connectivity now implies that any vertex lies on a cycle. For any $s \in C$, any two cycles $\gamma_1,\gamma_2$ from $s$ to $s$ begin with the same edge (otherwise some multiples of these two cycles would produce two different paths from $s$ to $s$ of the same length). Because of strong connectivity any edge in $C$ can be prolonged to a cycle, no vertex in $C$ has two outgoing edges. Thus, each vertex has outdegree exactly $1$. It follows that $C$ is a disjoint union of cycles. Since $C$ is connected, it is a cycle graph.
%
\end{proof}
Any path from the initial vertex $s_{\bullet}$ to a vertex $s$ with $\tau(s)=1$ passes through promising states only and has the form $\gamma=u_1 v_2 u_2 \cdots v_k u_k,$ where $k \geq 1$, $v_i\in \Sigma_k$, $u_i\in \Sigma_k^*$, $u_i$ are paths contained entirely in a strongly connected component $C_i$, while $v_i$ are simple edges between strongly connected components. Since $S^+/\mathord\sim$ is strictly ordered by the relation $\prec$, $k$ does not exceed the number of components $\abs{S^+/\sim}$. Furthermore, since $C_i$ are cycle graphs, any $u_i$ has the form $u_i=\tilde{u}_i w_i^{l_i}$, $l_i\geq 0$, where $w_i$ is a (shortest) cycle in $C_i$ and $\abs{\tilde u_i}$ is bounded by the size of $C_i$.
Hence, any such path $\gamma$ is uniquely determined by the following data:
\begin{enumerate} 
\item integer $1 \leq k \leq \abs{S^+/\sim}$;
\item the edges $v_i$, $2\leq i \leq k$;
\item the paths $\tilde u_i$, and $w_i$, $1\leq i \leq k$;
\item the multiplicities $l_i$, $1\leq i \leq k$.
\end{enumerate}
There are only finitely many choices of $k$; $v_i$ for $2\leq i \leq k$; and $\tilde{u_i}$, $w_i$ for $1\leq i \leq k$ as above. 
Hence every path from $s_{\bullet}$ to a vertex $s_t$ with $\tau(s_t)=1$ is of one of finitely many forms $$\gamma=\tilde{v}_1 w_1^{l_1} \tilde{v}_2 w_2^{l_2}\cdots\tilde{v}_k w_k^{l_k}, \quad l_i\geq 0,$$ where $\tilde{v}_1=\tilde{u}_1$, $\tilde{v}_i=v_i\tilde{u}_i$, $2\leq i \leq k$. This ends the proof of the proposition.
\end{proof}
\begin{proposition}\label{stemlesswineglass} Let $(a_n)_{n \geq 0}$ be an automatic sequence with values in $\{0,1\}$. Then either \begin{enumerate} \item  there exist integers $0 \leq l <m$ and integers $p, r_1, r_2$ such that $0\leq p <k^l$, $0\leq r_i < k^m$, $i=1,2$ such that $r_1\equiv r_2\equiv p \bmod{k^l}$, $$a_{k^l n +p}=a_{k^m n + r_1}=a_{k^m n + r_2}, \quad n\geq 0$$ and there exists some $n\geq 0$ such that $a_{k^l n+p}=1$  \end{enumerate} or \begin{enumerate} \item[(ii)]  the set $A=\{n\geq 0 \mid a_n=1\}$ is $k$-very sparse. \end{enumerate}\end{proposition}

\begin{proof} Let $\mathcal{A}$ be an automaton that produces the sequence $(a_n)$ when reading digits starting from the least significant one. Of course, we may additionally assume that every state is reachable from the initial state, i.e.\ for every $s\in S$ there exists $v\in \Sigma_k^*$ such that $\tilde{\delta}(s_{\bullet},v)=s$. Apply Proposition \ref{stemlesswineglass} to the automaton $\mathcal{A}$. If (i) holds, choose $v$ such that $\tilde{\delta}(s_{\bullet},v)=s$ and put $l=|v|$, $m=|v|+|w_1|$, $p=[v]_k^R$, $r_i=[w_iv]_k^R$, $i=1,2$. (Recall that $u^R$ denotes the reversal of the word $u$.) This choice satisfies all the required properties since $a_{k^l n +p} =1$ if and only if $\tilde{\delta}(s_{\bullet},v(n)_k^R)=1$ if and only if $\tilde{\delta}(s,(n)_k^R)=1$ which by a similar argument is equivalent to both  $a_{k^m n +r_1} =1$ and  $a_{k^m n +r_2} =1$. 

If (ii) holds, the claim is obvious.
\end{proof}

We are now ready to prove our main result.
\begin{proof}[Proof of Theorem \ref{thm:Structure-Auto}]
	Suppose that $(a_n)_{n \geq 0}$ is a sparse but not very sparse $k$-automatic sequence. By Proposition \ref{stemlesswineglass}\comment{I just love your reference names. --JK}, there we may find integers $l,t,p,s_1,s_2$ such that for any $n$, we have 
	$$a_{k^l n + p} = a_{k^{l}(k^d n + s_1) + p} =  a_{k^{l} (k^d n + s_2) + p}.$$
Moreover, there is some $n_0$ with $a_{k^l n_0 + p} = 1$. Inductive argument shows that $a_n = 1$ for any $n$ of the form $k^{l} \bra{ k^{td} + k^{(t-1)d} s_{i_1} + \dots + s_{i_t} } + p $ where $t \in \NN,\ i_{1},\dots,i_{t} \in \{0,1\} 1$. The set of such $n$ is \IPS.
\end{proof}

\begin{lemma}\label{lem:v-sparse-growth}
	Let $E \subset $ be very sparse of rank $r$, but not of rank $r-1$. Then, $\abs{E \cap [N]} = \Theta(\log^r N)$. Moreover, $\abs{E \cap [M,M+N) } = O(\log^r N)$, where the implied constant depends only on $E$.
\end{lemma}
\begin{proof}
	It will suffice to deal with the basic very sparse sets, given by
	\begin{equation}\label{eq:110}
		E = \set{ [w_0 u_1^{l_1} w_1 u_2^{l_2} \dots u_r^{l_r} w_r]_k }{ l_1, \dots, l_r \in \NN_0}.
	\end{equation}
	Replacing $k$ with a sufficiently large power $k^a$, we may assume that $\abs{w_i} = 1$ and $\abs{u_i} = 1$ for all $i$. It is clear from counting the number of possible choices of $l_i$ in \eqref{eq:110} that $\abs{E \cap [N]} \ll \log^r N$.
	
For the lower bound, note that all words $w_0 u_1^{l_1} w_1 u_2^{l_2} \dots u_r^{l_r} w_r$ for $l_i \in \NN_0$ are distinct. Indeed, suppose there were $(l_1,l_2,\dots,l_r) \neq (l_1',l_2',\dots,l_r') \in \NN_0^r$ such that 
	\begin{equation}\label{eq:111}
 w_0 u_1^{l_1} w_1 u_2^{l_2} \dots u_r^{l_r} w_r
	= w_0 u_1^{l_1'} w_1 u_2^{l_2'} \dots u_r^{l_r'} w_r.
	\end{equation}
In particular, $\sum_{i=1}^r l_i = \sum_{i=1}^r l'_i$. Assuming without loss of generality that $l_i$ is larger lexicographically than $l_i'$, we may find an index $j$ such that $l_j > l_j'$ while $l_i = l_i'$ for $i < j$. Comparing the two words in \eqref{eq:111}, we find that $u_j = w_j = u_{j+1}$. But then $E$ is a rank $r-1$ very sparse set, contradicting the choice of $r$. The bound $\abs{E \cap [N]} \gg \log^r N$ now follows by the same counting argument as before.

For the additionaly part, extending the interval $[M,M+N)$ by at most a constant factor, we may assume that $N = k^A$ is a power of $k$, and that $k^A \mid M$. An integer $n = [w_0 u_1^{l_1} w_1 u_2^{l_2} \dots u_r^{l_r} w_r]_k \in E \cap [M,M+N)$ is now uniquely determined by its last $A$ digits, which may be easily bounded by $(A+1)^r$ by choosing for each $j$ the digit where $w_j$ appears, if at all, among the $A$ end positions.
\end{proof}

By similar methods, we may estimate the growth of the elements of the \IPS-set produced in Proof of Theorem \ref{thm:Structure-Auto}.

\begin{corollary}\label{honorarywineglass} Let $(a_n)_{n \geq 0}$ be an automatic sequence with values in $\{0,1\}$ and let $\nu(N)=|\{0\leq n \leq N-1 \mid a_n=1\}|$. Then either 
\begin{enumerate}
	 \item there exist $c,\alpha>0$ such that $\nu(N) \geq c N^{\alpha}$ for large enough $N$ 
\end{enumerate}
 or 
\begin{enumerate}
	\item[(ii)]  there exist $c,l>0$ such that $\nu(N) \leq c (\log N)^l$ for large enough $N$. \end{enumerate}
\end{corollary}


\comment{The following remark is very far-fetched and should perhaps be eliminated.}

\begin{remark}\label{bloodylandmower} It might be interesting to point out that in \cite{BellCoonsHare-2014} it has  recently been proven that a regular  sequence  $(b_n)_{n\geq 0}$ that is unbounded satisfies $|b_n| \geq c \log n$ for some $c\geq 0$ and infinitely many $n\geq 0$. In relation to this result, note that if $(a_n)$ is an automatic sequence with values in $\{0,1\}$, then the sequence $(\nu(N))_{N\geq 0}$, $\nu(N)=|\{0\leq n \leq N-1 \mid a_n=1\}|$ is regular. (This follows from \cite[Theorem 16.4.1]{AlloucheShallit-book}.) This allows us perhaps to speculate whether a more general result on the rate of growth of regular sequences can be proven. A na\"ive question that comes to mind is the following: for an unbounded regular sequence $(b_n)_{n\geq 0}$, is it true that there exist constants $c,\alpha>0$ and an integer $l\geq 0$ such that $$\limsup_{n\to \infty} b_n/cn^{\alpha}(\log n)^l=1?$$ \end{remark}


\begin{remark}\label{kubekkawey} Let $a \geq 1$ be an integer. Then the notions of $k$-very sparse and $k^a$-very sparse sets coincide. Indeed, basic $k^a$-very sparse sets are obviously basic $k$-very sparse sets, and $k$-very sparse can be written as finite unions of basic $k^a$-very sparse sets. Alternatively, this is a consequence of the fact that the notions of $k$-automatic and $k^a$-automatic sequences coincide together with Corollary \ref{honorarywineglass} which characterizes $k$-very sparse sets as exactly the sets $A$ which can be obtained as $A=\{n\geq 0 \mid a_n=1\}$ for a $k$-automatic sequence $(a_n)$ with values in $\{0,1\}$ and such that the number of elements of $A$ that are $\leq N$ grows at most as $c  (\log N)^l$ for some $c>0$, $l\geq 0$.\end{remark}

\begin{proposition}\label{jeszczejednoespresso} Let $A$ be an infinite $k$-very sparse set. Then there exist integers $n\geq1$, $r\geq 0$, $p\geq 1$, and words $v_1,\ldots,v_p,w,u\in \Sigma_k^*$, $w\neq \epsilon$ such that $$A \cap (n\Z+r) = \bigcup_{i=1}^p \{[v_iw^lu]_k \mid l\geq 0\}.$$ 
\end{proposition}

\begin{proof} The set $A$ is a finite union of basic $k$-very sparse sets of the form $$A^{(m)}=\{ [v_1^{(m)}(w^{(m)}_1)^{l_1}v_2^{(m)}\cdots v^{(m)}_t(w^{(m)}_t)^{l_t}]_k \mid l_1,\ldots,l_t\geq 0\}, \quad m=1,\ldots,p'$$ (where $t$ depends on $m$). We begin by making this representation of $A$ as a union of basic sets $A^{(m)}$ slightly more convenient to work with. We can clearly assume that all the $w_i^{(m)}$'s except possibly for the final one $w_t^{(m)}$ satisfy $w_i^{(m)} \neq \epsilon$. Let $M$ be the lowest common multiple of the lengths of all the $w_i^{(m)}$'s in all the basic sets $A^{(m)}$ (note that at least one $w_i^{(m)}\neq \epsilon$ since $A$ is infinite). By writing the basic sets as unions of several smaller basic sets depending on what the residue of $l_i \pmod{M/|w_i^{(m)}|}$ is (and adequately enlarging $v_i^{(m)}$), we may assume that all the $w_i^{(m)}$'s have the same length $M$ (or $w_i^{(m)}=\epsilon$; the latter can only happen if $i=t$). We claim that we can further assume that the lengths of all the $v_i^{(m)}$'s are divisible by $M$. To this end, we modify the representation of each $A^{(m)}$ as a set of the form $A^{(m)}=\{ [v_1^{(m)}(w^{(m)}_1)^{l_1}v_2^{(m)}\cdots v^{(m)}_t(w^{(m)}_t)^{l_t}]_k \mid l_1,\ldots,l_t\geq 0\}$ by appropriately adjusting $v_i^{(m)}$ and  $w_i^{(m)}$. We pass from the right hand side to the left hand side and whenever some $v=v_i^{(m)}$ has length that is not divisible by $M$, we write the corresponding $w=w_{i-1}^{(m)}$ as $w=w' w''$ with $|w''|+|v|$ divisible by $M$ and we change $w^l v=(w' w'')^l v$ to $w' (w'' w')^{l-1} w'' v$ (and finally, we add another basic set that corresponds to the choice $l=0$). Eventually, we tap $v_1^{(m)}$ with a few extra initial zeroes if necessary. By replacing $k$ with $k^M$, we may assume that all the $w_i^{(m)}$ have length zero or one. We may further assume that whenever two consecutive letters $w_{i-1}^{(m)}$ and $w_i^{(m)}$ are equal $w_{i-1}^{(m)}=w_i^{(m)}=w$ or whenever $i=t$, $w_{i-1}^{(m)}=w, w_{i}^{(m)} = \epsilon$, the word $v=v_i^{(m)}$ in not of the form  $v=w^r, r\geq 1$. Let $M'$ denote the maximal length $T=\max_{i,m} |v_i^{(m)}|$. Without loss of generality we may assume that among all the $A^{(m)}$ it is the $A^{(1)}$ that has the biggest number of $j$ such that $w_j^{(m)}\neq \epsilon$ (this number is equal to $t-1$ or $t$  depending whether $w_t^{(m)}=\epsilon$ or not). Now take $u=(w^{(m)}_1)^{T+1}v^{(m)}_2\cdots v^{(m)}_t(w^{(m)}_t)^{T+1}$, $s=|u|$, $r=[u]_k$ and $n=k^s$. Then $A\cap (n\Z+r)$ involves only those $A^{(m)}$ with all $w_{i}^{(m)}=w_{i}^{(1)}$ (except possibly for the final $\epsilon$ in either $A^{(m)}$ or $A^{(1)}$). Therefore, $A\cap (n\Z+r)$ is a finite union of sets of the form $B^{(m)}=\{[v^{(m)}w^l u]_k\mid l\geq 0\}$ for some $v^{(m)},w,u\in \Sigma_k^*$, $w\neq \epsilon$.
\end{proof}

\begin{proposition}\label{automverysparse} If the set $\{k^l \mid l\geq 0\}$ is not generalised polynomial, then neither is any infinite $k$-very sparse set.
\end{proposition}
\begin{proof} Note first that if the set $E=\{k^l \mid l\geq 0\}$ is not generalised polynomial, then neither is any set of the form $E_t=\{k^{tl} \mid l\geq 0\}$, $t\geq 1$ since $E=\bigcup_{j=0}^{t-1}k^jE_t$.
\comment{There was a conflict of notation with $P$.}

Assume that an infinite $k$-very sparse set is generalised polynomial. Since the class of generalised polynomial sets contains all arithmetic progressions and is closed under finite intersection, Proposition \ref{jeszczejednoespresso} allows us to assume that
$$A = \bigcup_{i=1}^p \{[v_iw^lu]_k \mid l\geq 0\}$$
 for some $p\geq 1, v_1,\ldots,v_p,w,u\in \Sigma_k^*$, $w\neq \epsilon$. Let $s=|u|, t=|w|$ and note that $$[v_iw^lu]_k=[u]_k+k^s[w]_k \frac{k^{tl}-1}{k^t-1}+[v_i]_kk^{tl+s}.$$ 
 
 Let $P$ be a generalised polynomial such that $A=\{n\in \N  \mid P(n)=0\}$ and assume further that $P$ is a restriction of a generalised polynomial of a real variable that has no further zeroes in $\R \setminus \N$. (To this end, replace $P(n)$ by $P(n)^2+(n-\lfloor n \rfloor)^2$.)  Then an easy computation shows that the polynomial $$Q(n)=P(k^s(n-[w]_k)/(k^t-1)+[u]_k)$$ has as its zero set $$B=\{n\in \N \mid Q(n)=0\}=\bigcup_{i=1}^p \{b_i k^{tl}\mid l\geq 0\}$$ where $b_i = [w]_k+(k^t-1)[v_i]_k$, $i=1,\ldots,p$. The set $C=\{n\in N\mid b_1 n \in B\}$ is also generalised polynomial and it has the form $$C=\bigcup_{i=1}^{p} \{c_i k^{tl}\mid l\geq 0\}$$ with $c_1=1$ (and $c_i=b_ik^{tl_i}/b_1$ where $l_i\geq 0$ is the smallest integer such that $b_1$ divides $b_i k^{tl_i}$. (If there is no such integer, the corresponding term is not present.)
Let $m\geq 1$ be such that $c_i <k^{tm}$ for $i=1,\ldots,p$. Replacing the set $\{c_i k^{tl}\mid l\geq 0\}$ by the union $$\{c_i k^{tl}\mid l\geq 0\}=\cup_{j=0}^{m-1} \{c_i k^{tj} k^{mtl}\mid l\geq 0\}$$ and replacing $k$ by $k^{mt}$, we may assume that $$C=\bigcup_{i=1}^{p} \{c_i k^{l}\mid l\geq 0\}$$ with $c_1=1$ and $1\leq c_i<k^2$. Consider the set $D=\{n\in C \mid n\equiv 1 \pmod{k^2-1}\}$. The set $D$ is generalised polynomial and an element $c_i k^l$ can be an element of $D$ only if $c_i \equiv 1 \pmod{k^2-1}$ or $c_i \equiv k  \pmod{k^2-1}$. Since $1\leq c_i \leq k^2-1$, this gives $c_i=1$ or $c_i=k$ and whether the latter possibility is realized or not, we have $D=\{k^{2l} \mid l\geq 0\}$. This is a contradiction with our remark that no set of the form $P_t=\{k^{tl} \mid l\geq 0\}$, $t\geq 1$, is generalised polynomial (note that during the proof we have replaced $k$ by its power).
\end{proof}

\begin{proposition}\label{automshiftedipset} Let $(a_n)_{n\geq 0}$ be a $k$-automatic sequence with values in $\{0,1\}$. Assume that for every $w\in \Sigma_k^*$ there is an integer $n\geq 0$ such that $w$ is a factor of $(n)_k$ and $a_n=1$. Then the set $F=\{n\geq 0 \mid a_n=1\}$ contains a set of the form $E+N$, where $N\geq 0$ is an integer and $E\subset \N$ is an \IP-set.\end{proposition}

\begin{proof}
Let $\mathcal{A}=(S,s_{\bullet},\delta,\{0,1\},\tau)$ be a $k$-automaton that produces $(a_n)_{n\geq 0}$ by reading the digits of $n$ \emph{starting with the least significant one}. We will denote the word $u=0\ldots 0\in \Sigma_k^*$ with $n$ zeroes by $u=0^n$. We begin by proving the following claim.

{\bf Claim}: There exist states $s,s'\in S$ such that $\tau(s)=1$, and integers $l\geq 1$, $k^{l-1}\leq m<k^l$ such that for $v=[m]^R_k$, $u=0^l$  
we have $\tilde{\delta}(s,u)=s'$, $\tilde{\delta}(s,v)=s$, $\tilde{\delta}(s',u)=s'$, $\tilde{\delta}(s',v)=s$. This is portrayed below: 
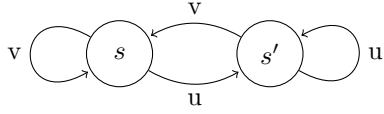
\begin{center}\begin{tikzpicture}[shorten >=1pt,node distance=2cm, on grid, auto] 
   \node[state] (s)   {$s$}; 
   \node[state] (s') [right=of s] {$s'$}; 
  \tikzstyle{loop}=[min distance=6mm,in=210,out=150,looseness=7]
  
    \path[->] 
    
    (s) edge [loop left] node {v} (s)
          edge [bend right] node [below]  {u} (s');
          
 \tikzstyle{loop}=[min distance=6mm,in=30,out=-30,looseness=7]
 \path[->]
    (s') edge [bend right] node [above]  {v} (s)
          edge [loop right] node  {u} (s');
\end{tikzpicture}\end{center} \begin{proof}[Proof of the claim.] Let $n=|S|$ be the number of states in $\mathcal{A}$.  Let $S=\{s_1,\ldots,s_n\}$ and let $T=\{t_1,\ldots,t_p\}\subset S$ denote the set of states $s$ in $S$ such that $\tau(s)=1$. We first show that there is a state $s\in T$ such that if we start at $s$, and then follow in the automaton the path of $n$ zeroes, we arrive at a state from which we can return to $s$ in such a way that the final step does not involve the digit zero. More formally, there exists a word $w=w_{s}\ldots w_0\in \Sigma_k^*$ with $w_i\in \Sigma_k$, $w_0\neq 0$ and $\tilde{\delta}(s,0^n w)=s$. 

To prove this, consider a path $\gamma$ in the automaton that is constructed in the following manner. The path $\gamma$ is a concatenation of paths $\gamma_1,\ldots,\gamma_n$. The path $\gamma_1$ is constructed as follows. Start at the state $s_1$, and go to the state $t_1$ (if possible), and then follow a path of $n$ zeroes. If the state $t_1$ was not accessible from $s_1$, ignore this step. Next, starting from the state you are currently in, go to state $t_2$ (if possible), and then follow a path of $n$ zeroes. Repeat for the remaining states $t_3,\ldots,t_p$. We call the path obtained in this way $\gamma_1$ and identify it with a word in $\Sigma_k^*$. To construct $\gamma_2$, follow the same procedure, but imagine that you begin at the state $\tilde{\delta}(s_2,\gamma_1)$ (the state that one would reach from $s_2$ following the path $\gamma_1$). Similarly, the paths $\gamma_i\in \Sigma_k^*$ for $1\leq i \leq n$ are obtained in the same manner starting at the state $\tilde{\delta}(s_i,\gamma_1\ldots\gamma_{i-1})$. Finally, $\gamma=\gamma_1\ldots\gamma_n$. By assumption, there is an integer $n\geq 0$ such that $[n]_k^R$ contains $\gamma$ as a factor and $\tau(\tilde{\delta}(s_{\bullet},[n]_k^R))=1$. Let $s=\tilde{\delta}(s_{\bullet},[n]_k^R)$. We claim that if we start at the state $s$ and follow a path of $n$ zeroes, we arrive at a state from which we can return to $s'$ in such a way that the final step does not involve the digit zero.  This is so because passing the path given by $[n]_k^R$ from $s_{\bullet}$ to $s$ we must have already visited $s$, then followed a path of $n$ zeroes, and then returned to $s$. Indeed, since $[n]_k^R$ contains $\gamma$ as a factor, we can write $[n]_k^R=u\gamma v$. The state $\tilde{\delta}(s_{\bullet},u)$ is equal to some $s_i$, and the path $\gamma$ was constructed in such a way that we visited $s$ using the path $\gamma_i$ (the state $s$ was reachable at that time since we have reached it also at a later time). 

Since $S$ has only $n$ states, we know that there exist $0\leq i<j\leq n$ such that $ \tilde{\delta}(s,0^i)=\tilde{\delta}(s,0^j)$. Let $p\geq n$ be an integer divisible by $(j-i)$ and let $s'$ be the state $s'=\tilde{\delta}(s,0^p)$. Since $p$ is divisible by $(j-i)$, we have $s'=\tilde{\delta}(s',0^{p})=\tilde{\delta}(s,0^{2p})$. We know that $s'$ is equal to $ \tilde{\delta}(s,0^q)$ for some $0\leq q\leq n$. Hence there exists a word $w$ such that $\tilde{\delta}(s',w)=s$ whose final digit is nonzero. Take $v=(0^p w)^{p}$, $u=0^{p(p+|w|)}$, $l=p(p+|w|)$. Since the final digit of $v$ is nonzero, we have $v=[m]^R_k$ for some $1\leq m <k^{l}$. This satisfies the demands of the claim. \end{proof}

To finish the proof, recall that the state $s$ of the claim satisfies $\tilde{\delta}(s_{\bullet},[n]_k^R)=s$. In the notation of the claim, take $N=n$, and integer $h$ such that $k^{h-1}\leq n<k^h$, $n_i=mk^{l(i-1)+h}$, and $E=\FS(n_i)$. One easily sees from the claim that for $q\in E+N$ we have $\tilde{\delta}(s_{\bullet},[q]_k^R)=s$, and so $a_q=1$. Indeed, for every $q\in E+N$, the word $[q]_k^R$ takes the form $[q]_k^R=[n]_k^R w_1\cdots w_g$ with $g\geq 0$, $w_i\in \{u,v\}$ and $w_g=v$.\end{proof}

\section{Explicit examples}\label{sec:Examples}

\comment{I think we want to move this to the end. --JK}

In this section, we discuss some special classes of sequences, for which we are able to prove non-automaticity. Our arguments here do not rely on Theorem \ref{thm:GP-vs-IPS}, and are generally speaking simpler. Not all sequences under consideration are generalised polynomials, but many are.

\subsection*{Polynomial sequences}

We begin with an example reminiscent of Proposition \ref{prop:poly=>not-auto}, except in the sparse setting.

\begin{proposition}\label{prop:adense-special-gen-poly=>not-auto}
	Let $f \colon \NN_0 \to \{0,1\}$ be a function of the form 
	$$f(n) = 
	\begin{cases}
	1, & \text{if } { \fpa{\p(n)} < \e(n) }, \\
	0, & \text{otherwise,}
	\end{cases}
	$$ where $p$ is a polynomial and $\e(n) \to 0$. Then, $f$ is not automatic, unless it is eventually periodic. Moreover, if $f$ is not eventually periodic, then there is no \IPS-set $E$ with $f(n) = 1$ for $n \in E$.
\end{proposition}
\begin{remark}
	If $1/\e(n)$ is a generalised polynomial, then so is $f(n)$.
\end{remark}

\begin{proof}
	It is clear that $f$ is periodic precisely when all coefficients of $p$ are rational, except possibly for the constant term, so let us assume that this is not the case. Dilating if necessary, we may assume that the leading coefficient of $p$ is irrational. If $f(n) = 0$ for all but finitely many $n$, then we are done, so suppose also that this does not hold. Write
	$$
		p(x) = \sum_{i=0}^d a_i x^i,
	$$
	where $a_d \in \RR \setminus \QQ$.
	
	For the sake of contradiction, suppose that $f$ were $k$-automatic. Then, by the Pumping Lemma \ref{lem:pumping}, there exist words $u,v,w \in \Sigma_k^*$ such that $f([vu^tw]_{k}) = 1$ for all $t \in \NN_0$. Hence, there are constants $r,s \in \QQ$, $r \neq 0$, such that $f(r k^t + s) = 1$. Replacing $f(n)$ with $f(r n + s)$, we may assume that $r = 1,\ s = 0$. 

	Denote by $\Delta^{t}_m \colon \RR[x] \to \RR[x]$ the operator given by $(\Delta^t_m q)(x) = q(mx) - m^t q(x)$. Then $\Delta^{t}_m (x^j) = (m^j - m^t) x^j$, and in particular we have
	\begin{equation}
	\label{eq:934}
	\Delta^{d-1}_k \Delta^{d-1}_k \dots \Delta^{1}_k \Delta^{0}_k p (x) = K a_d x^d,
	\end{equation}
where $K = \prod_{t=0}^{d-1} (k^d - k^t) \in \NN$. Expanding the left hand side of \eqref{eq:934} for $x = k^t$ and letting $t \to \infty$, we find
	\begin{equation}
	\label{eq:935}
	\fpa{ K a_d k^{td} } = \fpa{ \sum_{j} A_j p(k^{t+j})  } \leq \sum_{j} A_j \e(k^{t+j}) = o(1),
	\end{equation}
	where $A_j \in \ZZ$ are constants (dependent only on $k$ and $d$). This is only possible if the base $k$ expansion of $K a_d$ has only finitely many non-zero terms, whence $a_d \in \QQ$ contrary to the previous assumption. 

To prove the the additional part of the statement, suppose that $(N_t)_{t\geq 1}$, $(n_\a)_{\a \in \cF}$ were the sequences such that $f(n) = 1$ for $n \in \FS(n_i;N_t)$. Again, we assume that the leading term of $p$ is irrational.

Expand $p(x-y) = \sum_{j=0}^d b_j(x) y^j$, where $b_j(x) \in \RR[x]$; in particular $b_d(x) = (-1)^d a_d \in \RR \setminus \QQ$. For a fixed $\a \in \cF$ and $t \geq \max \a$ we have
	$$ \fpa{ p(N_t - n_\a) } = \fpa{ \sum_{j=0}^d b_j(N_t) n_\a^j } \leq \e(N_t - n_\a).$$
	 Passing to the limit $t \to \infty$ along a subsequence such that $\fp{b_j(N_t)}$ converges to some $a_j' \in [0,1]$, we conclude that 
	$$ \sum_{j=0}^d a_j' n^j_\alpha \equiv 0 \pmod \ZZ. $$
	Picking $\b \in \cF$ with $\abs{\b} = d$, and using an inclusion-exclusion argument we find that
	$$ a_d' d! \prod_{i \in \b} n_i =  \sum_{\a \subseteq \b} (-1)^{d-\abs{\a}} \sum_{j=0}^d a_j' n^j_\alpha \equiv 0 \pmod \ZZ, $$
	which contradics irrationality of $ a_d' = (-1)^d a_d$.
	
	\comment{I think the formula above works, but could you (=JB) double-check that? --JK}
\end{proof}

\subsection*{Linear recurrence sequences}

In contrast to Example \ref{ex:lin-rec-is-gp}, we show that the set of values of a linear recurrence sequence is not automatic, except for the trivial examples.

\begin{proposition}
	Let $(a_m)_{m\geq 0}$ be an $\N$-valued sequence satisfying a linear recurrence of the form
	\begin{equation}
	\label{eq:380} 
		a_{m + D} = \sum_{i=1}^{D} c_i a_{m+D-i}, \quad  m\geq 0
	\end{equation}
	with integer coefficients $c_i$. Suppose that $E = \set{ a_m}{ m \in \NN_0}$ is $k$-automatic for some $k$. Then $E$ is a finite union of the following standard sets: linear progressions $\set{ am +b}{ m \in \NN_0}$, exponential progressions $\set{ a k^{ d m} + b}{m \in \NN_0}$ and finite sets.
\end{proposition}
\begin{proof}
	We first claim there exists a representation of $E$ as a finite union
	\begin{equation}\label{eq:381}
	E = \bigcup_{i=1}^{K_{\text{lin}}} L_i \cup \bigcup_{i=1}^{K_{\text{poly}}} P_i \cup \bigcup_{i = 1}^{K_{\text{exp}}} E_i \cup F
	,\end{equation}
	where $F$ is finite, $L_i = \set{ a_i m + b_i}{m \in \NN_0}$ are arithmetic progressions, $P_i = \set{ p_i(m) }{m \in \NN_0}$ are sets of values of polynomials $p_i(x) \in \ZZ[x]$ with $\deg p_i \geq 2$, $E_i$ have exponential growth in the sense that $\abs{E_i \cap [N]} \ll \log N$.

	In order to prove this claim, begin by noting that any restriction of $(a_m)$ to an arithmetic progression $a^{(h,r)}_m = a_{hm + r}$ obeys some (minimal length) linear recurrence 
	$$
	a^{(h,r)}_{m + D'} = \sum_{i=1}^{D'} c^{(h,r)}_i a^{(h,r)}_{m+D'-i}, \quad m\in \NN_0
	$$
	with $D'=D'(h,r) \leq D$. Moreover, there exists a choice of $h$ such that each of that each $c^{(h,r)}_m$ is either identically zero or non-degenerate, in the sense that the associated characteristic polynomial $q^{(h,r)}(x) = x^{D'} - \sum_{i=1}^{D'} c^{(h,r)}_{i} x^{D'-i}$ has no pair of roots $\lambda,\mu \in \CC$ such that $\lambda/\mu$ is a root of unity
	(see e.g.\ \cite[Thm 1.2]{EverestPoortenShparlinskiWard-book} for a much stronger statement).  
	Hence, for the purpose of showing the existence of representation \eqref{eq:381}, we may assume  that $(a_m)$ is non-degenerate. Suppose also that $D$ is minimal, and let $\lambda_1,\dots,\lambda_r$ be the roots of $q(x) = x^D - \sum_{i=1}^D c_{i} x^{D-i}$, with $\abs{\lambda_1} \geq \abs{\lambda_2} \geq \dots$. Note that $\abs{\lambda_1} \geq 1$.
	
	If $\abs{\lambda_1} > 1$, then by the result of Evertse \cite{Evertse-1984} and van der Poorten and Schlickewei \cite{PoortenSchlickewei-1991} (see \cite[Thm 2.3]{EverestPoortenShparlinskiWard-book}), we have $a_m = \abs{\lambda_1}^{m + o(m)}$ as $m \to \infty$. Hence, $E$ has exponential growth, and we are done.
	
	Otherwise, if $\abs{\lambda_1} =1$, then for all $j$ we have $\abs{\lambda_j} = 1$ or $\lambda_j=0$. Kronecker's theorem \cite{Kronecker-1857} (or a standard Galois theory argument) shows that if $\lambda$ is an algebraic integer all of whose conjugates have absolute value $1$, then $\lambda$ is a root of unity. \comment{I saw that in a MathOverflow post, but have little idea how standard it is --- I suck at Galois theory. See:\\ http://math.stackexchange.com/questions/4323/are-all-algebraic-integers-with-absolute-value-1-roots-of-unity}\\ Using the general formula for the solution of a linear recurrence, we may write
	$$
		a_m = \sum_{j=1}^r \lambda_j^m p_j(m) = \sum_{j=1}^r b_j(m) p_j(m),
	$$
	where $p_j(x)$ are polynomials and $b_j(m)$ are periodic. Splitting $\NN_0$ into arithmetic progressions where $b_j(m)$ are constant, we conclude that $E$ is a finite union of sets of values of polynomials. This again produces a representation of the form \eqref{eq:381}. 
	
	Such a representation is not unique. Splitting $P_i$ into a finite number of sub-progressions and discarding those which are redundant, we may assume that $P_i \cap L_j = \emptyset$ for any $i,j$. Likewise, we may assume that $E_i \cap L_j = F \cap F_j =\emptyset$ for any $i,j$. Fix one representation subject to these restrictions. The set 
	
	$$E' =\bigcup_{i=1}^{K_{\text{poly}}} P_i \cup \bigcup_{i = 1}^{K_{\text{exp}}} E_i \cup F = E \setminus \bigcup_{i=1}^{K_{\text{lin}}} L_i $$ is again $k$-automatic; it will suffice to show that $E'$ is a union of the standard sets mentioned above.
	
	We claim that $K_{\text{poly}} = 0$, i.e.\ the representation of $E$ uses no polynomial progressions of degree $ \geq 2$. Suppose for the sake of contradiction that $P = \set{p(m)}{m \in \NN_0}$ appears is one of the sets $P_i$, and write $p(m) = \sum_{i=0}^d c_i m^i$, where $c_i \in \ZZ$. Replacing $p(m)$ with $p(m+r)$ for a suitably chosen $r \in \NN_0$, we may assume that $c_i > 0$ for $0 \leq i \leq d$. For sufficiently large $t$, we have $p(k^t) = [u_d 0^{t - t_0} u_{d-1} 0^{t-t_0} u_{d-2} \dots 1_1 o^{t-t_0} u_{0}]_{k}$, where $t_0$ is a constant and $u_{i}$ is the base $k$ expansion of $c_i$, padded by $0$'s so as to have $\abs{u_i}=t_0$. Since $p(k^t) \in E'$, from the Pumping Lemma \ref{lem:pumping} it follows that there is $l \in \NN$ such that for any $s_1,\dots,s_d\in\NN$ it holds that
	$$
		n(s_1,\dots, s_d) := [u_{d} 0^{ls_d} u_{d-1} 0^{ls_{d-1}} \dots u_1 0^{ls_1} u_0]_{k} \in E'.
	$$
	For sufficiently large $S$ and an absolute constant $\delta$ to be determined later, consider the set 
	$$
		Q(S) = \set{n(s_1,\dots,s_d)}{ s_i \in \NN, s_1 + \dots + s_d = S,\ s_2 + \dots + s_d \leq \delta S }, 
	$$
 	and put $N(S) := n(S-d+1,1,\dots,1) = \min Q(S)$. Note that $N(S) = k^{l S + O(1)}$ and that $\max Q(S) = N(S) + O(N(S)^\delta)$. 
 	For fixed $T_0$ and $T \to \infty$, we shall consider the cardinality of the set $ Q(T_0,T) = \bigcup_{T_0 \leq S \leq T} Q(S)$. By an elementary counting argument, we find 
 	\begin{equation}
 	\label{eq:438}
 	\abs{ Q(T_0,T)} \gg T^{d} \gg T^2. 
 	\end{equation}
 	
 	To obtain an upper bound, we separately estimate $\abs{ Q(S) \cap P_i}$ and $\abs{Q(T_0,T) \cap E_j}$ for each $i,j$.
 	
 	Suppose that $n, n' \in Q(S) \cap P_i$ with $n' > n$, so in particular $n = p_i(m)$ and $n' = p_i(m')$ for some $m,m' \gg N(S)^{1/\deg p_i}$. We then have the chain of inequalities:
 	$$
 		N(S)^\delta \gg n'-n = p_i(m') - p_i(m) \geq \min_{x \in [m,m']} \abs{ p_i'(x) } \gg N(S)^{\frac{ \deg p_i - 1}{\deg p_i}}, 
 	$$
	which is a contradiction for sufficiently large $S$, provided that $\delta < \frac{\deg p_i - 1}{\deg p_i}$ (which will hold if we put $\delta = \frac{1}{3}$). Thus, $\abs{ Q(S) \cap P_i} \leq 1$. 
	
	As for $Q(T_0,T) \cap E_j$, from the bounds on growth of $E_j$ we immediately have
	\begin{equation}
 	\label{eq:439}
 	\abs{ \bigcup_{T_0 \leq S \leq T} Q(T_0,T) \cap E_i} \ll \abs{E_i \cap [2N(T)]} \ll T.
 	\end{equation}

	 In total, using \eqref{eq:438} and \eqref{eq:439} we find that
 	\begin{equation}
 	\label{eq:440}
 	\abs{ Q(T_0, T) } \leq \sum_{S=T_0}^{T} \sum_{i=1}^{K_{\text{poly}}} \abs{ Q(S) \cap P_i} + \sum_{i=1}^{K_{\text{exp}}} \abs{ Q(T_0, T) \cap E_i} + O(1) \ll T,
 	\end{equation}
	 contradicting the previously obtained bound. It follows that indeed $K_{\text{poly}} = 0$.
	
	Since $E'$ contains no polynomial or linear progressions, we have $\abs{E' \cap [N]} \ll \log N$. It follows from Theorem \ref{thm:Structure-Auto} (or more accurately Corollary \ref{honorarywineglass}) \comment{Make sure to state it so that it holds true!}\ that $E'$ must be very sparse of rank $1$ and base $k$. Since all basic very sparse sets of rank $1$ are of the form described in the formulation of the theorem, we are done. 
\end{proof}

\subsection*{IP-rich sequences}

In light of Theorem \ref{thm:Structure-Auto}, any sparse but not very sparse automatic set contains an \IPS- set. It turns out that a slightly stronger condition of containing (a translate of) an \IP-set leads to interesting consequences. We will say that a $\{0,1\}$-valued sequence $f$ is \IPrich\ (resp. \IPPrich, \IPSrich) if there exists an \IP-set (resp. \IPP-set, \IPS-set) $E$ such that $f(n) = 1$ for $n \in E$.

In contrast to Theorem \ref{thm:Structure-Auto}, there are sparse automatic sequences which are neither very sparse nor \IPPrich\ (the interested reader will have no difficulty finding examples). However, many natural examples of sparse but not very sparse automatic sequences turn out to be \IPrich.

\begin{example}\label{ex:B-free}
	Let $k \geq 2$, and let $\cB \subset \Sigma_k^*$ be a finite set of ``prohibited'' words of length $\leq t$. A word $u \in \Sigma_k^*$ is $\cB$-free if $u$ contains no $b \in \cB$ as a substring, and accordingly $n \in \NN_0$ is $\cB$-free if its expansion base $k$ is $\cB$-free. Denote
		$$
		f_\cB(n) =
		\begin{cases}
			0, &\text{if $n$ is $\cB$ free}, \\
			1, &\text{otherwise.}
		\end{cases}
	$$
\begin{enumerate}
\item The function $f_\cB$ is $k$-automatic.
\item If $\cB \neq \emptyset$, then $f_\cB$ is sparse.
\item If $ \sum_{b \in \cB} k^{-\abs{b}} < 1/8t$, then $f_\cB$ is not very sparse.
\item If each $b \in \cB$ contains at least two non-zero digits, then $f_\cB$ is \IPrich.
\item If some $b \in \cB$ consists only of $0$'s, then $f_\cB$ is not \IPrich
\end{enumerate}
\end{example}

\begin{proof}
\begin{enumerate}
\item It is not difficult to explicitly describe the $k$-automaton which computes $f$.
\item A an elementary computation shows that if $t = \max_{b \in \cB} \abs{b}$, then 
$$\frac{1}{N} \sum_{n \leq N} f(n) \leq 2 \bra{1 - \frac{\abs{\cB}}{k^t}}^{\floor{ \frac{\log N}{ t \log k} }} = o(1), \text{ as } N \to \infty.$$

\item Construct an undirected graph $G = (V,E)$ (allowing loops), where $V = \Sigma_k^{t}$, and $\{u,v\} \in E$ if $uv$ and $vu$ are both $\cB$-free. If $u_1, u_2, \dots, u_r$ is a path in $G$, then $u_1 u_2 \dots u_r$ is $\cB$-free. If $G$ contains a path $u_1,w,u_2$ of length $3$ with $u_1 \neq u_2$, then for any $i_1,\dots,i_r \in \{1,2\}$, the word $u_{i_1} w u_{i_2} w \dots u_{i_r} w$ is $\cB$-free, and hence $f_\cB$ is not very sparse. Thus, it remains to check that $G$ contains a length $3$ path with distinct endpoints; for the sake of contradiction suppose this is not the case.

On one hand, any vertex $u \in V$ can have at most $1$ neighbour (including itself if $\{u,u\}$ is an edge), and it is an easy exercise in extremal combinatorics to show that $\abs{E} \leq \abs{V} = k^{t}$. On the other hand, given $b \in \cB$, the number of pairs $(u,v) \in V^2$ such that $b$ appears in $uv$ or $vu$ is $< 4t k^{2t -\abs{b}}$, so 
$$
	\abs{E} \geq \frac{1}{2} k^{2t} - 4 t k^{2t} \sum_{b \in \cB} k^{-\abs{b}}
	>  \frac{1}{4} k^{2t} > k^{t},
$$
(note that assumption implies $k^t \geq 8$), which is the sought contradiction.

\comment{Here, path is NOT assumed to be non-self-intersecting!}
\item Let $n_i = k^{it}$. Then $f(n) = 1$ for each $n \in \FS(n_i)$.
\item If $E$ is an \IP-set, then some $n \in E$ is divisible by $k^t$, and $f(n) = 0$.
\end{enumerate}

\end{proof}	
	
\begin{example}\label{ex:Fib-IPrich}
The sequence
	$$
		f_{\mathrm{Fib}}(n) =
		\begin{cases}
			0, &\text{if binary expansion of $n$ contains $11$}, \\
			1, &\text{otherwise,}
		\end{cases}
	$$
is $2$-automatic, sparse, not very sparse, and \IPrich.
\end{example}

\begin{example}[Baum-Sweet]\label{ex:BS-sparse-IPrich}
	The Baum-Sweet sequence (\cite{BaumSweet}), given by
	$$
		f_{\mathrm{BS}}(n) =
		\begin{cases}
			0, &\text{if binary expansion of $n$ contains $10^l1$ with $l \in 2\NN_0+1$},  \\
			1, &\text{otherwise,}
		\end{cases}
	$$
is $2$-automatic, sparse, not very sparse, and \IPrich.
\end{example}


While the property of being \IPrich\ is arguably somewhat rare among sparse automatic sequences, it is possible to ensure that a sequence is \IPPrich\ under some relatively mild assumptions.

\begin{lemma}\label{lem:IP+-rich-if-pattern-rich}
	Suppose that $f$ is a $k$-automatic sequence, with the property that each $u \in \Sigma_k^*$ is a factor of $(n)_k$ for some $n \in \NN_0$ with $f(n) = 1$. Then $f$ is \IPPrich.
\end{lemma}
\begin{proof}
	Immediate reformulation of Proposition \ref{automshiftedipset}.
\end{proof}

Restricting to \IPrich\ sequences, we can easily prove an analogue of Proposition \ref{prop:adense-special-gen-poly=>not-auto} for generalised polynimials. 

\begin{proposition}\label{prop:IP-auto-vs-special-gp}
	Let $f$ be a sparse, \IPrich\ automatic sequence. Let $q$ be a generalised polynomial, $\e(n) \geq 0$ with $\e(n) \to 0$ as $n \to \infty$, and let $g$ be given by 
	$$g(n) = 
	\begin{cases}
	1, & \text{ if }  \fpa{q(n)} < \e(n), \\
	0, & \text{ otherwise.} 
	\end{cases} 
	$$
	Suppose that the set $R = \set{n \in \NN_0}{f(n) = g(n)}$ is $\mathrm{IP}^*$. Then, there are infinitely many $n$ with $q(n) \in \ZZ$.
\end{proposition}

\begin{remark}
	If $1/\e(n)$ is a generalised polynomial, then so is $g(n)$.
\end{remark}
\begin{corollary}\label{cor:special-sparse-gp=>not-auto}
	With notation in Proposition \ref{prop:IP-auto-vs-special-gp}, if $g(n)$ is $k$-automatic for some $k \geq 2$, then at least one of the following holds:
\begin{enumerate}
	\item there exists $b \in \Sigma_k^*$ such that $f(n) = 0$ whenever $b$ is a factor of $(n)_k$;
	\item there are infinitely many $n \in \NN_0$ with $q(n) \in \ZZ$.
\end{enumerate}
\end{corollary}

\begin{proof}[Proof of Proposition \ref{prop:IP-auto-vs-special-gp}]
	For the sake of contradiction, suppose that $R$ is $\mathrm{IP}^*$. By definition, there exists an IP-set on which $f$ takes the value $1$. The intersection of an IP-set and $\mathrm{IP}^*$-set contains an IP-set, so we may find an IP-set $(n_\a)_{\a \in \cF} $ such that $f(n_\a) = g(n_\a) = 1$ for $\a \in \cF$.

	To proceed, we need an algebraic relation between $(q(n_\a))_{\a \in \cF}$, which is analogous to the vanishing of sufficiently high discrete derivatives of polynomials (cf.\ Proof of Proposition \ref{prop:adense-special-gen-poly=>not-auto}). Such relation is provided by \cite[Theorem 2.42]{BergelsonMcCutcheon-2010} in the language of limits along idempotent ultrafilters. In the language of \IP-rings, using standard methods to translate between the two contexts, the theorem implies that there exist $\lambda \in \RR$, $d \in \NN_0$ (the degree of $q$), as well as families of \IP-rings $\cG_1,\ \cG_2(\b_1),\ \dots,\  \cG_d(\b_1,\dots,\b_{d-1})$ ($\b_i \in \cF$), such that 
\begin{equation}
\label{eq:763}
		\sum_{\emptyset \neq \a \subseteq [d]} (-1)^{\abs{\a}} q\bra{ \sum_{i \in \a} n_{\b_i} } = - \lambda;
\end{equation}	
	whenever $\b_1 \in \cG_1,\ \b_2 \in \cG_2(\b_1),\ \dots,\ \b_d \in \cG_d(\b_1,\dots, \b_{d-1})$.
	\comment{I think we could require that $\cG_i(\b_1,\dots,\b_i)$ consists of all $\a$ in some $\cG_i*$ with $\a > N(\b_1,\dots,\b_{i-1})$ for some function $N(\dots)$. I'm not absolutely certain this stronger statement is true. The stronger statement is more in line with the philosophy of \IP-limits. However, we don't need it, and it is easier to see that the weaker statement holds.}
	
	
	In particular, for $\b_1,\dots,\b_d$ as above we have
\begin{equation}
	\fpa{\lambda} \leq \sum_{\emptyset \neq \a \subseteq [d]} \e\bra{ \sum_{i \in \a} n_{\b_i} },
\end{equation}
\label{eq:764}
	so passing to the \IP-limit with $\b_d,\dots,\b_1$ along the corresponding \IP-rings, we conclude that $\lambda \in \ZZ$. Similarly, fixing $\b_1 \in \cG_1$ and passing to the limit with $\b_d,\dots,\b_2$, we conclude from
\begin{equation}
\label{eq:765}
	\fpa{ q(n_{\b_1}) } 
	 \leq \sum_{\substack{\emptyset \neq \a \subseteq [d],\\ \a \neq \{1\}} } \e\bra{ \sum_{i \in \a} n_{\b_i} }
\end{equation}
that $q$ takes integer values on the (infinite) \IP-set $\set{n_{\b} }{\b \in \cG_1}$.
\end{proof}

\begin{remark}
	In light of Proposition \ref{prop:IP-auto-vs-special-gp}, it is interesting to ask which generalised polynomials $q(n)$ have the property that $q(n) \in \ZZ$ for at most finitely many $n$. While it is a hopeless task to classify all such generalised polynomials, we point out that many natural examples have this property. For instance, all polynomials in Theorem \ref{thm:equidistr-gen-poly} are such. 
	\comment{Example (4) is misquoted, need to look it up and correct. --jk}
\end{remark}

\subsection*{Heisenberg example}
We close with an example of a fairly explicit class of sparse sequences for which we can prove non-automaticity.

\begin{example}\label{ex:Heisenberg} There exists an absolute constant $c > 0$ such that the following holds. Let $1,\a,\b \in \RR$ be algebraic numbers linearly independent over $\QQ$, and $\e(n) \to 0$ with $\e(n) \gg n^{-c}$. Then, the sequence given by
	$$f(n) = 
	\begin{cases}
	1, & \text{ if }  \fpa{ n\a \floor{n \b} } < \e(n), \\
	0, & \text{ otherwise} 
	\end{cases} 
	$$
	is not $k$-automatic.
%
\end{example}

A key tool which we will use is the quantitative equidistribution theorem for nilrotations, whose special case cited below we use as a black box.

\begin{theorem}[Green-Tao, \cite{GreenTao2012}]\label{thm:GreenTao}
	Let $g(n)$ be a polynomial sequence on a nilmanifold $G/\Gamma$ with \Malcev\ basis $\cX$, where $G$ is connected and simply connected. Then, there exists a constant $C > 0$, such that for any $\e > 0$ and any $N$, one of the following holds:
	\begin{enumerate}
	\item\label{item:GT:1} for each $x \in G/\Gamma$, there exists $n \in [N]$ such that $d_{\cX}( g(n)\Gamma, x) \leq N^{-\e}$; 
	\item\label{item:GT:2} there exists a horizontal character $\eta$ with $\norm{\eta} \ll N^{C\e}$ such that 
	$$
		\max_{0 \leq n < N} \fpa{ \eta \circ g(n+1) - \eta \circ g(n) } \ll N^{-1+C\e}.
	$$
	\end{enumerate}
\end{theorem}
\comment{}

Strictly speaking, Green and Tao work with a rather different notion of quantitative equidistribution; the cited result follows immediately from applying Theorem 1.16 to in \cite{GreenTao2012} to $\delta$-equidistribution of $g(n)\Gamma$ with respect to suitable bump function with support in the radius $N^{-\e}$ ball around $x$, where $\delta$ is a suitable small power of $1/N$.

For linear orbits $g(n) = a^n$, condition \eqref{item:GT:2} in Theorem \ref{thm:GreenTao} 
 asserst that the projection of $a$ to the torus $G/G_2\Gamma$ satisfies an approximate linear relation over $\ZZ$. A convenient criterion which can be used to rule that out is provided by the following classical result of Schmidt. As before, we only cite the special case which we shall use.

\begin{theorem}[Schmidt \cite{Schmidt1972}]\label{thm:Schmidt}
	Let $\alpha_1,\dots,\alpha_n$ be $n$ algebraic numbers linearly independent over $\QQ$, and let $\e > 0$. Then, for $k \in \ZZ^n$ we have
	$$
		\fpa{ \sum_{i=1}^n k_i \alpha_i } \gg \norm{k}^{-n-\e},
	$$ 
	where the implicit constant depends on $\alpha_1,\dots,\alpha_n$ and $\e$.
\end{theorem}

\begin{remark}
In fact, we only use the above with $\e = 1$, and any bound of the form  $\fpa{ \sum_{i=1}^n k_i \alpha_i } \gg \norm{k}^{-C}$ would work equally well for our purposes, where $C$ is a constant dependent at most on $n$. We note that the same result with $C$ dependent on the degrees of $\a_i$ is elementary. If we allowed the constant $c$ in Proposition \ref{ex:Heisenberg} to depend on $\a,\b$, then the weaker bound would suffice for our purposes.
\end{remark}


	

\begin{proof}[Proof of Proposition \ref{ex:Heisenberg}]
	Fix $k \geq 2$. We claim that for any sequence of digits $u \in \Sigma_k^*$, there exists $n \in \NN_0$ whose expansion base $k$ ends with $u$, and such that $f(n) = 1$. Once this is accomplished, it follows from Corollary \ref{cor:special-sparse-gp=>not-auto} that $f(n)$ is not $k$-automatic. Letting $r = [u]_k$ and $t \geq \abs{u}$, we are left with the task of finding $n$ such that $f'(n) := f(k^t n + r) = 1$. 
	
	For the sake of contradiction, suppose that $f'(n) = 0$ for all $n \in \NN_0$. Without loss of generality we may also assume that $\e(n)$ is decreasing. The value of the constant $c$ will be determined in the course of the argument.
	
	We can explicitly describe a nilmanifold on which $f(n)$ can be realised in a way analogous to Theorem \ref{thm:BergelsonLeibman}. Denote 
	$$[x,y,z] \equiv \begin{bmatrix}
		1 & x & z \\ 
		0 & 1 & y \\
		0 & 0 & 1 
	\end{bmatrix},$$
and define 
	$$
		G = \set{ [x,y,z] }{x, y, z \in \RR},\qquad 
		\Gamma \set{ [x,y,z] }{x, y, z \in \ZZ}.	
	$$
	Then, $G/\Gamma$ is the Heisenberg nilmanifold, with a natural choice of \Malcev\ basis $[1,0,0],\ [0,1,0],\ [0,0,1]$.
	
	The horizontal characters on $G/\Gamma$ take the form $\eta_k([x,y,z]) = k_1 x + k_2 y $ with $k = (k_1,k_2) \in \ZZ^2$ and $\norm{\eta_k} = \norm{k}_2$. Let $g$ be the polynomial sequence
	$$g(n) = [- n \a, n \b, 0] = [-\a,\b,0]^n [0,0,\a\b]^{\binom{n}{2}},$$ so that $\fp{g(n)} = [ \fp{-n\a}, \fp{n\b}, \fp{ n\a \floor{n\b} }]$. Put also $g'(n) = g(k^t n + r)$. 
	
	Let $F \colon G/\Gamma \to [0,1)$ be the map $F([x,y,z]) = \fpa{z}$ for $x,y,z \in [0,1)^3$. It follows from the way the metric on $G/\Gamma$ is defined that $F$ is Lipschitz\comment{(Admittedly, not super clear. However, I think I would like to leave it like that, because 1) it is intuitively clear that it should be so 2) this is an example, so we can afford to be slightly imprecise 3) the alternative is to go into the details of how the metric on $G/\Gamma$ is defined, and it's neither enlightening nor pleasant...)}. With this notation,
	\begin{equation}\label{eq:254}
	f(n) = \ifbra{ F(g(n)) < \e(n) } 	
	\end{equation}

	Let $N$ be a large integer, and put $z = [0,0,0]\Gamma$. There exists $\rho_N \gg \e(N)$, such that $F(B(z,\rho_N)) \subset [0,\e(N))$. Hence, if $d(z, g(n)\Gamma) < \rho_N$ for some $n \in [N] \cap (k^t \NN_0 +r)$, then $f(n) = 1$ contrary to the assumption. It now follows from Theorem \ref{thm:GreenTao} that there exists a horizontal character $\eta$ with $\norm{\eta} \ll \rho^{-A}_N \ll N^{Ac}$ such that 
	  \begin{equation}
		  \label{eq:255}	  
	  	\max_{0 \leq n < N} \fpa{ \eta( g'(n+1) ) - \eta(g'(n)) } \ll N^{1-Ac},	  
	  \end{equation}
	  where $A$ is an absolute constant. Picking $l \in \ZZ^2$ so that $\eta = \eta_l$, we conclude from \eqref{eq:255} and Theorem \ref{thm:Schmidt} that 
	  \begin{equation}
	    \label{eq:256}	  
	  	N^{-3Ac} \ll \norm{l}^{-2} \ll \fpa{ k^{t} ( l_1 \a + l_2 \b) } \ll N^{1-Ac}.
	  \end{equation}
Picking any $c < \frac{1}{4A}$ we reach the sought contradiction.
\end{proof}

\begin{remark}\label{remark:Heisenberg}
	The above argument is not specific to the sequence $\ifbra{ n\a \floor{n\b} < \e(n)}$; in fact, the term $n\a \floor{n\b}$ could be replaced by a fairly arbitrary generalised polynomial $q(n)$. The difficulty lies in the need to impose suitable ``quantative irrationality'' conditions of $q(n)$.
	
	In the situation of Example \ref{ex:Heisenberg}, this was easily accomplished since we had direct access to the representation of $q(n)$ related to an orbit on a nilmanifold. For more general sequences, one needs to resort to construction like in \cite{BergelsonLeibman2007}.
\end{remark}

\subsection*{Exponential sequences} In order to construct a generalised polynomial that takes value zero exactly on the set of Fibonacci numbers, we used the theory of continued fractions. In order to replace the Fibonacci numbers by a linear recurrence sequence of higher degree, we need to use higher dimensional diophantine approximation. We briefly review the topic below. Let $\theta$ be a point in $\R^d$, $d\geq 1$, and let $N$ denote a norm on $\R^d$. Let $$N_0(\theta)=\inf_{p\in \Z^d} N(\theta-p)$$ denote the distance from $\theta$ to a nearest lattice point. We say that an integer $q\geq 1$ is a \emph{best approximation} of $\theta$ with respect to the norm $N$ if $N_0(q\theta)<N_0(k\theta)$ for $1\leq k\leq q-1$. If $\theta \notin \Q^d$, there are infinitely many best approximations of $\theta$. We order them in an increasing sequence $(q_n)_{n\geq 0}$, $1=q_0<q_1<\ldots$ which we call the \emph{sequence of best approximations}. For $d=1$, the sequence is well known as the sequence of denominators of convergents of the continued fraction expansion of $\theta$ (in the latter sequence the first term possibly appears twice). For $d\geq 2$ and general $\theta\in \R^d$ there is no satisfactory theory of best approximations. However, the problem has been studied for specific choices of $\theta$, notably by Lagarias \cite{Lagarias-1982}, Chekhova-Hubert-Messaudi \cite{ChekhovaHubertMessaoudi-2001}, Chevallier \cite{Chevallier-2013}, and Hubert-Messaudi \cite{HubertMessaudi-2006} in the case when the coordinates of $\theta$ lie in some cubic number fields. We follow the presentation of \cite{HubertMessaudi-2006}. 

Consider the polynomial $p(x)=x^3-ax^2-bx-1$ with integer coefficients $a,b$ satysfying ($a\geq 0$ and $0\leq b \leq a+1$) or ($a\geq 2$ and $b=-1$). Assume further that the polynomial $p$ has a unique real root $\beta>1$ and a pair of complex conjugate roots $\alpha, \bar{\alpha}$. (The condition on $a$ and $b$ arises from the work of Akiyama \cite{Akiyama-2000} who classifies Pisot units of degree 3 satisfying the so-called Property (F) concerning finiteness of certain $\beta$-expansions.) Then $\beta$ generates an imaginary cubic Pisot field. We will show that the set $\{\lbr \beta^n \rbr \mid n\geq 0\} $ is generalised polynomial. To this end, we use the results of Hubert and Messaoudi who studied the best approximations of the point $\theta=(\beta^{-1}, \beta^{-2})$ \cite[Theorem 1]{HubertMessaudi-2006}. We need to begin, however, by the following basic result.

\begin{proposition}\label{przepraszamzetakpozno} Let $\beta$ be a Pisot number (i.e.\ a real algebraic integer $\beta>1$ such that all its Galois conjugates have absolute value smaller than one). Let $p=x^d-c_1x^{d-1}-\ldots-c_d$ be the minimal polynomial of $\beta$. Let $(R_n)_{n\geq 0}$ be a sequence of integers satisfying the linear recurrence relation $$R_{n+d}=c_1R_{n+d-1}+\ldots+c_dR_n, \quad n\geq 0.$$  Assume that $(R_n)_{n\geq 0}$ is not identically zero. Then the following conditions are equivalent: \begin{enumerate} \item The set $\{R_n\mid n\geq 0\}$ is generalised polynomial. \item The set $
\{\lbr \beta^n \rbr \mid n\geq 0\}$ is generalised polynomial. \end{enumerate} In particular the question whether the set $\{R_n\mid n\geq 0\}$ is generalised polynomial depends only on the recurrence relation but not on the initial values of $(R_n)_{n\geq 0}$.\end{proposition}

\begin{proof} From the form of the linear recurrence relation for $(R_n)_{n\geq 0},$ we see that $R_n=u\beta^n + o(1)$ for some $u\in \R$. Note that $u \neq 0$. In fact, otherwise $R_n=o(1)$ and since $R_n$ takes integer values we see that $R_n=0$ for large $n$. It follows that $R_n$ is identically zero (note that $c_d\neq 0$). 

Since $\beta$ is an algebraic integer, we know that $\sum_{\beta'} \beta'^n$ is an integer, the sum being taken over all the Galois conjugates $\beta'$ of $\beta$. Since $|\beta'|<1$ for $\beta'\neq \beta$, we see that $\beta^n = \lbr \beta^n \rbr + o(1).$ It follows that $$R_n = u \lbr \beta^n \rbr + o(1)$$ and the claim follows immediately from the following lemma. \end{proof}
\begin{lemma} Let $u\in \R$, $u\neq 0$ and let  $(R_n)_{n\geq 0} $ and $(S_n)_{n\geq 0} $ be integer valued sequences such that $$R_n=uS_n+o(1).$$ Then the set  $\{R_n\mid n\geq 0\}$ is generalised polynomial if and only if the set  $\{S_n\mid n\geq 0\}$ is generalised polynomial.
\end{lemma}

\begin{proof} It follows from the relation $R_n=uS_n+o(1)$ that for large $n$ an integer $m$ is of the form $m=S_n$ if and only if $\lbr u m \rbr = R_n$ and $\fpa{um}<|u|/2$. We immediately see from this that if the set $\{R_n\mid n\geq 0\}$ is generalised polynomial, then so is  $\{S_n\mid n\geq 0\}$. Since the argument is symmetric with respect to $(R_n)$ and $(S_n)$, we conclude the claim.\end{proof}

We will apply these results to a particular linear recurrence sequence $R_n$ studied by Hubert-Messaoudi.

\begin{theorem}[\cite{HubertMessaudi-2006}]\label{bambambabam} Let the polynomial $p(x)=x^3-ax^2-bx-1$ with roots $\beta, \alpha, \bar{\alpha}$ be as above. Define the sequence $(R_n)_{n\geq 0}$ defined by the recurrence formula $R_0=1$, $R_1=a$, $R_2=a^2+b$, $$R_{n}=aR_{n-1}+bR_{n-2}+R_{n-3}, \quad n\geq 3.$$ Then there exists a norm $N$ on $\R^2$ (called \emph{a Rauzy norm}) such that the sequence of best approximations of the point $\theta=(\beta^{-1}, \beta^{-2})$ with respect to the norm $N$ coincides except for finitely many elements with the sequence $(R_n)$. Furthermore, there exists a sequence of real numbers $(m_q)_{q\geq 1}$ satisfying the following properties:

\begin{enumerate}
\item\label{bambambabam.mapprox} There exists a constant $c>0$ such that for all $q\geq 1$ and $p\in \Z^2$ if $N(q\theta-p)<c$, then $$N(q\theta-p)=N_0(q\theta)=m_q.$$
\item\label{bambambabam.mmin} If $q<R_n$, then $m_{R_n}<m_q$.
\item\label{bambambabam.mwzor} For $n\geq 0$ we have $m_{R_n}=m_1|\alpha|^n$. 
\item We have $\liminf_{q\to \infty} m_q =0.$
\item\label{bambambabam.beta} We have $$\beta^n=R_n + \frac{b\beta+1}{\beta^2} R_{n-1} + \frac{1}{\beta} R_{n-2}, \quad n\geq 2.$$
\item\label{bambambabam.nwzor} The norm $N$ is given by the formula $$N(x)=|(\alpha+b/\beta)x_1+x_2/\beta|, \quad x=(x_1,x_2)\in \R^2.$$
\end{enumerate}
\end{theorem}
\begin{proof} This is proven in \cite[Theorem 1, Theorem 3, Corollary 2, Proposition 7, Lemma 4, Lemma 2, and formula (3)]{HubertMessaudi-2006}. The sequence $m_q$ is denoted $N(\delta(q))$ there.
\end{proof}
\begin{remark}
The choice of the initial values of $(R_n)$ might seem unnatural. Note however that it corresponds to $R_0=1$, $R_{-1}=R_{-2}=0$.
\end{remark}

We are now ready to prove the following theorem. 

\begin{theorem}\label{LAcrime} Let $p$ be a polynomial $p(x)=x^3-ax^2-bx-1$ with integer coefficients $a,b$ satysfying ($a\geq 0$ and $0\leq b \leq a+1$) or ($a\geq 2$ and $b=-1$). Assume that $p$ has a unique real root $\beta$ and that $\beta>1$. Then the set $\{\lbr \beta^n \rbr \mid n\geq 0\}$ is generalised polynomial.\end{theorem}

\begin{proof} By Proposition \ref{przepraszamzetakpozno}, instead of $\lbr \beta^n \rbr$, we may study the sequence $R_n$ of the previous theorem. It is easy to verify that the sequence $(R_n)$ is strictly increasing for $n\geq 6$. We will start by constructing a generalised polynomial $g \colon \Z \to \R$ that is increasing and which takes the value $g(q)=m_q^{-2}$ at points $q=R_n$ with $n$ sufficiently large. The sequence $R_n$ satisfies a third order linear recurrence with characteristic polynomial $p$ and hence $$R_n=u\beta^n+v\alpha^n+w\bar{\alpha}^n$$ for some $u,v,w\in \CC$. Hence, $\lim_{n\to \infty} (R_n-u\beta^n)=0$, and thus for $n$ large enough we have $$R_{n-1}=\lbr R_{n}/\beta\rbr, \quad R_{n-2}=\lbr  R_{n}/\beta^2\rbr.$$
Define a generalised polynomial $g$ by the formula $$g(q)= m_1^{-2}\left(q + \frac{b\beta+1}{\beta^2} \lbr q/\beta \rbr + \frac{1}{\beta}\lbr  q/\beta^2\rbr\right), \quad q\in\Z.$$
Then by Theorem \ref{bambambabam}.(\ref{bambambabam.mwzor}) and \ref{bambambabam}.(\ref{bambambabam.beta}), we have $g(R_n)=m_1^{-2}|\beta|^{n}=m_{R_n}^{-2}$ for large $n$. (Note that $|\alpha|=|\beta|^{-1/2}$.) Furthermore, $g$ is increasing. 

\comment{It is possible that we have a conflict of notation between $p$ --- polynomial and $p$ --- integer/pair of integers.}
Consider the set $$S=\{q\geq 1 \mid N_0(q\theta)<g(q)^{-1/2}\}.$$ Since the sequence $\{R_n \mid n\geq 0\}$ is up to finitely many terms the sequence of best approximations of $\theta$, applying Theorem \ref{bambambabam}.(\ref{bambambabam.mapprox}) and \ref{bambambabam}.(\ref{bambambabam.mmin}) and the fact that $g$ is increasing and $g(R_n)=m_{R_n}^{-2}$ tends to $0$ as $n\to \infty$, we see that $S$ coincides with the set $\{R_n \mid n\geq 0\}$ up to a finite set. In order to show that $S$ is generalised polynomial, it is enough to write the expression $N_0(q\theta)$ in terms of a generalised polynomial in $q$. We will do that under the additional assumption that $N_0(q\theta)$ is small. Assume that $N_0(q\theta)<(2\Im(\alpha))^{-1}$ and let $p=(p_1,p_2)\in \Z^2$ be such that $N_0(q\theta)=N(q\theta-p)$. By Theorem \ref{bambambabam}.(\ref{bambambabam.nwzor}), we have $$N(q\theta-p)=|(\alpha+b/\beta)(q\beta^{-1}-p_1)+(q\beta^{-2}-p_2)/\beta|<1/(2\Im(\alpha)).$$ By looking at the imaginary part of the expression, it follows that $|q\beta{-1}-p_1|<1/2$, and hence $p_1$ is uniquely determined as $p_1=\lbr q\beta^{-1}\rbr.$ Since $p$ minimises $N(q\theta-p)$, we also have $$p_2=\lbr \beta(\alpha+b/\beta)(q\beta^{-1}-p_1)+q\beta^{-2}\rbr.$$

Therefore we have \begin{align*}N_0(q\theta)\leq &|(\alpha+b/\beta)(q\beta^{-1}-\lbr q\beta^{-1}\rbr)+\\+&(q\beta^{-2}-\lbr \beta(\alpha+b/\beta)(q\beta^{-1}-\lbr q\beta^{-1}\rbr)+q\beta^{-2}\rbr)/\beta|\end{align*}
and  equality holds if furthermore $N_0(q\theta)<(2\Im(\alpha))^{-1}$. Call the expression on the right hand side of this inequality $h(q)$. Rewriting the  norm in $\CC$ in terms of the coordinates, we see that $h(q)^2$ is given by a generalised polynomial. Therefore the set $$T=\{q\geq 1 \mid h(q)^2<g(q)^{-1}\}$$ coincides with $S$ up to a finite set. Hence, the set $\{R_n \mid n\geq 0\}$ is generalised polynomial. 
\end{proof}

\section{Concluding remarks}

We close with some general remark an questions which naturally appeared during the work on this project.

To begin with, we note that the difficulty in understanding the behaviour of arbitrary (possibly sparse) generalised polynomials should not come as a surprise. Indeed, many deep questions can be succinctly phrased in terms of generalised polynomials being identically $0$, which suggest that sparse generalised polynomials are difficult to understand. We give one such example below.
\begin{conjecture}[Littlewood]
	Let $\a,\b \in \RR$ and $\e > 0$. Then, the generalised polynomial $f \colon \NN \to \{0,1\}$ given by
	$$
		f(n) = \ifbra{ n \fpa{n\a} \fpa{n \b} < \e }
	$$
	is not identically $0$.
\end{conjecture}
It is known by the work fo Einsiedler, Katok and Lindenstrauss \cite{EinsiedlerKatokLindenstrauss-2006} that the set of $(\a,\b)$ for which the above fails has Hausdorff dimension $0$. However, the conjecture is as yet unresolved, and is believed to be difficult.

\begin{question}
	Does there exist a generalised polynomial $q(n)$ which takes any value $\bmod{1}$ finitely many times and a non-increasing sequence $\e(n) \geq 0$, such that 
	$$\ifbra{ \fpa{q(n)} \leq \e(n)} $$ is an automatic sequence, not eventually $0$?
\end{question}

\begin{remark}
The answer is negative when $q(n)$ is a genuine polynomial and either $\e(n)$ is constant (Proposition \ref{prop:poly=>not-auto}) or $\e(n) \to 0$ as $n \to \infty$ (Proposition \ref{prop:adense-special-gen-poly=>not-auto}), or when $q(n) = n \a \floor{n \b}$ and $\e(n) = n^{-o(1)}$ with $\a,\b$ algebraic and linearly independent (Proposition \ref{ex:Heisenberg}; see also Remark \ref{remark:Heisenberg}). The condition that $q(n) \bmod{1}$ does not take the same value infinitely often is added to remove trivial examples such as $q(n) \in \QQ[n] + \RR$, or $q(n) = g(n) h(n)$ where $g(n)$ is any sparse ($\{0,1\}$-valued) generalised polynomial and $h(n) \to 0$ as $n \to \infty$ with $g(n) = 1$.
\end{remark}

By Theorem \ref{thm:main-dichotomy}, the question of whether or not there exist non-trivial sequences which are simultaneously generalised polynomial and automatic reduces to deciding if the set $\set{k^t}{t \in \NN}$ is generalised polynomial; we suspect it is not. Conversely, by Example \ref{ex:lin-rec-is-gp}, the set of Fibonacci numbers, which can be described (up to finitely many exceptions) as $\set{ \lbr \varphi^t \rbr }{ t \in \NN}$ where $\varphi = (1+\sqrt{5})/2$, is generalised polynomial. Similarly, in Theorem \ref{LAcrime} we show  the set $\set{ \lbr \beta^t \rbr }{ t \in \NN}$ is generalised polynomial, where $\beta$ is a certain non-totally real cubic Pisot unit. This suggests the following questions. 

\begin{question}\label{Q:Pisot}
	For which $\lambda > 1$ is the set $$E_\lambda := \set{\lbr\lambda^t \rbr}{t \in \NN}$$ generalised polynomial?
\end{question}

We note that the set of such $\lambda$ includes at least all the Pisot units of degree $2$, which is essentially shown in Example \ref{ex:lin-rec-is-gp}, and some Pisot units of degree $3$.

\comment{I've changed Pisot number of degree 2 to Pisot unit of degree 2 -- JB. OK --JK}

An intimately related (cf. Proposition \ref{przepraszamzetakpozno}) and perhaps more natural question is the following one.

\begin{question}
	For which linear recurrence sequences $a=(a_n)_{n\geq 0}$ with  values in $\NN_0$ is the set $$E_a:= \set{ a_n }{n\geq 0}$$ generalised polynomial?
\end{question}

Looking at the aforementioned examples from the perspective of recursive relations with non-constant coefficients (as in Remark \ref{lawnmower}), we may also ask the following question.

\begin{question}\label{Q:BApprox}
	Let $n_i$ be a sequence given by $n_0 = 0,\ n_1 = 1$ and the recursive relation
	$$n_{i} = a_i n_{i-1} + b_i n_{i-2},$$ with $a_i,b_i \in \ZZ$ bounded in absolute value. For which $a = (a_i)_i,\ b = (b_i)_{i}$ is the set 
	$$E_{a,b} = \set{n_i}{i \in \NN}$$ generalised polynomial?
\end{question}

While generalised polynomial sets with positive density can be adequately understood through the application of Theorem \ref{thm:BergelsonLeibman}, much less is known about the sparse regime, as many of the previous examples suggest. In particular, it seems highly plausible that any ``sufficiently sparse'' set is generalised polynomial. \comment{I suggest to eliminate the word highly, but then you're generally braver than I am -- JB.}

\begin{question}\label{Q:sparse=>gen-poly}
	Does there exist a function $h \colon \NN \to \NN$ with the property that for any sequence $(n_i)_{i\geq 1} \subset \NN$ with $n_{i+1} \geq h(n_i)$ the set $E = \set{n_i}{i \in \NN}$ is generalised polynomial?	
\end{question}
	
\begin{remark}
Example in Section 3.4 in \cite{BergelsonLeibman2007} shows how to construct a sparse generalised polynomial set containing $2^{2^i-1}$ for all $i \in \NN$, and can be generalised to other sequences of similar growth. It is however not clear how to ensure that unwanted elements do not fall into a set thus constructed.  
\end{remark}

\begin{remark}
As a tangential remark, we point out that in light of Example \ref{ex:lin-rec-is-gp}, it could be interesting to study a class of Fibonacci-automatic sequences, defined as the sequences produced by finite automatons which on input the expansion of $n$ in the Fibonacci base. 
	
	Any $n \in \NN_0$ has an expansion $(n)_{\text{Fib}} = u_d u_{d-1} \dots u_0$ with $u_i \in \{0,1\}$, such that $n = \sum_i u_i f_i$ where $f_i$ is the $i$-th Fibonacci number, and no two consecutive $u_i$'s take the value $1$. (Such a representation is known to be unique by Zeckendorf's Theorem). Inasmuch as this notion makes sense, the sequence of Fibonacci numbers provides an example of a non-trivial generalised polynomial sequence which is Fibonacci-automatic. 
\end{remark}

We finish by presenting a generalisation of Conjecture \ref{conjecture:main} to regular sequences. We call a function $f \colon \N_0 \to \Z$ a quasi-polynomial if there exists an integer $m\geq 1$ such that the sequences $f_j$ given by $f_j(n)=f(mn+j)$, $0\leq j \leq m-1$, are polynomials in $n$. We say that a function $f \colon \N_0 \to \Z$ is ultimately a quasi-polynomial if it coincides with a quasi-polynomial except at a finite set.

\begin{question}\label{Q:mainregular}
	Assume that the sequence $f \colon \N_0 \to \Z$ is both regular and generalised polynomial. Is it then true that $f$ is ultimately a quasi-polynomial? 
\end{question}

If $f$ takes only finitely many values, then all the polynomials inducing $f_j$ are necessarily constant, and so in this case the question coincides with Conjecture \ref{conjecture:main}. If one wanted to reduce the regular case to the automatic case, it seems likely that a more precise information about the rate of growth of regular sequences along the lines of Remark \ref{bloodylandmower} might prove useful. \comment{If we eliminate the referenced remark, this needs to change as well -- JB.}


\bibliographystyle{alpha}
\bibliography{bibliography}

\end{document}